\documentclass[11pt, a4paper]{article}
\usepackage{fancyhdr}
\usepackage[british]{babel}

\usepackage{amssymb, amsmath, amsthm, amsfonts, enumerate}
\usepackage{amsmath,setspace,scalefnt}
\usepackage[usenames,dvipsnames,svgnames,table]{xcolor}
\usepackage{graphicx,tikz,caption,subcaption}
\usetikzlibrary{patterns}
\usetikzlibrary{shapes}
\usepackage{fullpage}
\usepackage{hyperref}
\usepackage{enumitem,verbatim}
\usepackage{algorithm} 
\usepackage{algpseudocode} 
\usepackage{todonotes} 
\usepackage{multicol}
\usepackage{float}

\usepackage[margin=2.5cm]{geometry}

\definecolor{gray}{rgb}{0.25, 0.25, 0.25}

\newtheorem{theorem}{Theorem}[section]
\newtheorem{lemma}[theorem]{Lemma}
\newtheorem{cor}[theorem]{Corollary}

\newtheorem{question}[theorem]{Question}
\newtheorem*{theorem*}{Theorem}

\theoremstyle{definition}
\newtheorem{subclaim}{Subclaim}

\newenvironment{poc}{\begin{proof}[Proof of claim]}{\end{proof}}

\newcounter{propcounter}

\theoremstyle{plain}
\newtheorem{claim}[theorem]{Claim}

\newtheorem{prop}[theorem]{Proposition}
\theoremstyle{definition}
\newtheorem{rmk}[theorem]{Remark}
\theoremstyle{definition}

\theoremstyle{definition}

\theoremstyle{definition}
\newtheorem{defn}[theorem]{Definition}
\theoremstyle{definition}

\theoremstyle{definition}

\newcommand\ex{\ensuremath{\mathrm{ex}}}

\newcommand{\eps}{\varepsilon}

\newcommand{\cS}{\mathcal{S}}

\newcommand{\cP}{\mathcal{P}}

\newcommand{\Ho}{\H{o}}

\newcommand{\emref}[1]{\emph{\ref{#1}}}

\newcommand{\makenote}[2]
{
	
	\smallskip
	
	\noindent
	
	\fbox{
		\begin{minipage}{0.95\textwidth}
			
			\def\temp{#1}
			\ifx\temp\empty
			\def\forlabel{$\bullet$}
			\else
			\def\forlabel{{\bfseries #1:}}
			\fi
			
			\begin{itemize}[label = \forlabel]
				#2
			\end{itemize}
		\end{minipage}
	}
	
	\smallskip
}

\title{Crux, space constraints and subdivisions}
\author{
	Seonghyuk Im
	\thanks{Department of Mathematical Sciences, KAIST, South Korea \& Extremal Combinatorics and Probability Group (ECOPRO), Institute for Basic Science(IBS), South Korea. Email: {\tt seonghyuk@kaist.ac.kr}. Supported by the POSCO Science Fellowship of POSCO TJ Park Foundation and IBS-R029-C4.}
	\and
	Jaehoon Kim\thanks{Department of Mathematical Sciences, KAIST, South Korea. Email: {\tt jaehoon.kim@kaist.ac.kr}. Supported by the POSCO Science Fellowship of POSCO TJ Park Foundation.}
	\and
	Younjin Kim
	\thanks{Extremal Combinatorics and Probability Group (ECOPRO), Institute for Basic Science (IBS), Daejeon, South Korea, Email: {\tt \{younjinkim,hongliu\}@ibs.re.kr}. Supported by IBS-R029-C4.}
	\and
	Hong Liu\footnotemark[3]
}

\begin{document}
	\maketitle
	\begin{abstract}
	For a given graph $H$, its subdivisions carry the same topological structure. The existence of $H$-subdivisions within a graph $G$ has deep connections with topological, structural and extremal properties of $G$. One prominent example of such a connection, due to Bollob\'{a}s and Thomason and independently Koml\'os and Szemer\'edi, asserts that the average degree of $G$ being $d$ ensures a $K_{\Omega(\sqrt{d})}$-subdivision in $G$. Although this square-root bound is best possible, various results showed that much larger clique subdivisions can be found in a graph for many natural classes. We investigate the connection between crux, a notion capturing the essential order of a graph, and the existence of large clique subdivisions. This reveals the unifying cause underpinning all those improvements for various classes of graphs studied. Roughly speaking, when embedding subdivisions, natural space constraints arise; and such space constraints can be measured via crux.
	
	Our main result gives an asymptotically optimal bound on the size of a largest clique subdivision in a generic graph $G$, which is determined by both its average degree and its crux size.  As corollaries, we obtain
	\begin{itemize}
	    \item a characterisation of extremal graphs for which the square-root bound above is tight: they are essentially disjoint unions of  graphs having crux size linear in $d$;
	    
	    \item a unifying approach to find a clique subdivision of almost optimal size in graphs which do not contain a fixed bipartite graph as a subgraph;
	    
	    \item and that the clique subdivision size in random graphs $G(n,p)$ witnesses a dichotomy: when $p = \omega(n^{-1/2})$, the barrier is the space, while when $p=o( n^{-1/2})$, the bottleneck is the density.
	   	\end{itemize}
	\end{abstract}

	\section{Introduction}

	For a graph $H$, a subdivision of $H$ (or an $H$-subdivision) is a graph obtained by replacing each edge of $H$ by internally vertex-disjoint paths. 
	As edges are simply replaced by paths, a subdivision of $H$ has the same topological structure with $H$. This notion builds a bridge between graph theory and topology. Studies on the existence of certain subdivisions in a given graph $G$ provide deep understanding on various aspects of $G$.
	For example, the cornerstone theorem of Kuratowski~\cite{Kuratowski30} in 1930 completely characterises planar graphs by proving that graphs are planar if and only if they do not contain a subdivision of either $K_5$, the complete graph on five vertices, or $K_{3,3}$, the complete bipartite graph with three vertices in each class.

	What conditions on graphs $G$ guarantee an $H$-subdivision in them? 
 	A fundamental result of Mader \cite{Mader1967} in 1967 states that a large enough average degree always provides a desired subdivision. 
 	Namely, for every $t\in \mathbb{N}$, there exists a smallest integer $f(t)$ such that every graph $G$ with the average degree at least $f(t)$ contains a subdivision of $K_t$.
 	He further conjectured that $f(t)=O(t^2)$.
	This conjecture was verified in the 90s by Bollob\'{a}s and Thomason \cite{Bollobas1998} and independently by Koml\'os and Szemer\'edi \cite{Komlos1996}. In fact, $f(t)=\Theta(t^2)$; the lower bound was observed by Jung \cite{Jung1970} in 1970: consider the $n$-vertex graph which is a disjoint union of $5n/t^2$ copies of $K_{\frac{t^2}{10}, \frac{t^2}{10}}$. A clique subdivision must be embedded in a connected graph; this example, though may have arbitrary large order $n$, is essentially the same as one copy of $K_{\frac{t^2}{10}, \frac{t^2}{10}}$, which does not contain a $K_t$-subdivision. Indeed, at least ${t/2\choose 2}$ many edges are subdivided in any $K_t$-subdivision in $K_{\frac{t^2}{10}, \frac{t^2}{10}}$, which would require around $t^2/8$ vertices on one side. In other words, apart from the obvious ``degree constraint'' from the average degree, there is also some ``space constraint'' forbidding a $K_t$-subdivision.
	
	\subsection{Crux and subdivisions}
	From the extremal example above, it is then natural to wonder if $G$ does not structurally look like $K_{\frac{t^2}{10}, \frac{t^2}{10}}$, can we find a larger clique subdivision? Indeed, Mader \cite{Mader1999} conjectured that every $C_4$-free graph with average degree $d$ contains a subdivision of $K_{\Omega(d)}$. Towards this conjecture, K\"uhn and Osthus \cite{Kuehn2004} proved that such a graph contains a subdivision of $K_{\Omega(\frac{d}{\log^{12} d})}$. K\"uhn and Osthus \cite{Kuehn2002, Kuehn2006} also considered $K_{s,t}$-free graphs and Balogh, Liu and Sharifzadeh \cite{Balogh2015} considered $C_{2k}$-free graphs with $k\geq 3$. 
Recently, Mader's conjecture was fully resolved by Liu and Montgomery~\cite{Liu2017}. 
Furthermore, they proved that for every $t \geq s \geq 2$, there exists a constant $c=c(s, t)$ such that if $G$ is $K_{s,t}$-free and has the average degree $d$, then $G$ has a subdivision of a clique of order $cd^{s/2(s-1)}$.

Note that a $C_4$-free graph with the average degree at least $d$ must have at least $\Omega(d^2)$ vertices as the maximum number of edges of an $n$-vertex $C_4$-free graph is $O(n^{3/2})$, hence providing enough space to put a $K_{\Omega(d)}$-subdivision with $O(d^2)$ vertices. Similarly, the number $d^{s/2(s-1)}$ also matches with the conjectured extremal number of $K_{s,t}$. Thus, all these $H$-free conditions relax the ``space constraints''. This suggests that `the essential order' of the graph $G$, rather than structural $H$-freeness, is an important factor for the size of the largest clique subdivision. 

How do we define `the essential order' of a graph? Simply taking a largest component is not sufficient. We can add a few edges to the disjoint union of $K_{O(t^2),O(t^2)}$ to make it connected, but those added edges are not useful for finding a $K_t$-subdivision unless we add enough edges to significantly increase the average degree. Haslegrave, Hu, Kim, Liu, Luan and Wang~\cite{Haslegrave2021} recently introduced a notion of the crux, which captures the `essential order' of a graph. We write $d(G)$ for the average degree of $G$.
	
	\begin{defn}[Crux]
		Let $\alpha>0$ and $G$ be a graph. A subgraph  $H\subseteq G$ is an \emph{$\alpha$-crux} if $d(H)\ge \alpha \cdot d(G)$. Let $c_{\alpha}(G)$ be the order of a smallest $\alpha$-crux in $G$, that is:
		$$c_{\alpha}(G)=\min\{n: \exists H\subseteq G \mbox{ s.t. } |H|=n, d(H)\ge \alpha d(G)\}.$$
	\end{defn}
 We will write simply $c(G)$ when $\alpha=1/100$; the choice of $1/100$ here is not special and can be replaced with any small number. Roughly speaking, the crux of a graph is large when the edges are relatively uniformly distributed.
 
 Let us see some examples of graphs whose crux size is much larger than their average degree (see~\cite{Haslegrave2021} for more details): $(i)$ $K_{s,t}$-free graphs $G$ with $s,t\ge 2$, satisfy $c_\alpha(G)\ge \frac{\alpha^{s/(s-1)}}{2t}d(G)^{s/(s-1)}$; $(ii)$ a $\frac{d}{r}$-blow-up $G$ of a $d$-vertex $r$-regular expander graph for a sufficiently large constant $r$ satisfies $c(G)=\Omega(d^2)$ and $d(G)=d$; and $(iii)$ there are graphs for which there is an exponential gap between its crux size and average degree: a simple application of isoperimetry inequalities on hypercubes shows that for $d$-dimension hypercube $Q^d$, $c_\alpha(Q^d)\ge 2^{\alpha d}$.
 
 \medskip
 
 Our main result reads as follows. It implies in particular that the space constraint, measured by crux size, is a deciding factor for the size of the largest clique subdivision in a graph. 
			
		\begin{theorem}\label{main_thm}
		There exists an absolute constant $\beta>0$ such that the following is true.
		Let $G$ be a graph with $d(G)=d$. Then $G$ contains a $K_{\beta t/(\log\log t)^6}$-subdivision where 
		$$t=  \min\Big\{d, \sqrt{\frac{c(G)}{\log c(G) }}\Big\}.$$
	\end{theorem}

Theorem~\ref{main_thm} asymptotically confirms a conjecture of Liu and Montgomery~\cite{Liu2017}. The bound above is optimal up to the multiplicative $(\log\log t)^6$ factor: the $d$-blowup of a $d$-vertex $O(1)$-regular expander satisfies $c(G)=\Theta(d^2)$ and the largest clique subdivision has order $d/\sqrt{\log d}$ (see \cite{Liu2017} for more details).

Our approach utilises the recently-developed sublinear expander theory. The notion of sublinear expander was first introduced by Koml\'os and Szemer\'{e}di \cite{Komlos1994, Komlos1996} in the 90s. There has been a series of advancements in the theory of sublinear expanders, resulting in a wide range of applications: for cycle embeddings, see~\cite{Fernandez2022,Haslegrave2021} and most notably~\cite{Liu2020} for the resolution of Erd\Ho s and Hajnal's odd cycle problem from 1966; for graph Ramsey goodness, see~\cite{Haslegrave2021b}; for an enumeration problem, see~\cite{Kim2017}; for minors, subdivisions and immersions, see~\cite{Haslegrave2021a,Liu2017,Liu2020a}; and for other sparse substructures, see~\cite{Fernandez2022a}.

Our proof makes use of the structures introduced in Liu and Montgomery's work~\cite{Liu2017} on Mader's conjecture on subdivisions and in Kim, Liu, Sharifzadeh and Staden's work~\cite{Kim2017} on Koml\'os's conjecture on Hamiltonian subsets, and expansion properties of sublinear expanders developed in Liu and Montgomery's work on Erd\Ho s-Hajnal's odd cycle problem~\cite{Liu2020}. We further employ an idea of iterative use of expanders from Haslegrave, Kim and Liu's study~\cite{Haslegrave2021a} on sparse minor embeddings to bootstrap an initial weaker bound to a near-optimal one by considering three nested layers of expanders with different expansion properties.

\begin{rmk}
Our main result is another illustration of the \emph{replacing average degree by crux} paradigm proposed in Haslegrave, Hu, Kim, Liu, Luan and Wang~\cite{Haslegrave2021}, which says roughly that for results about sparse graph embeddings, if the bound on the size of the substructure is depending on $d(G)$, then one might be able to replace it with the crux size $c(G)$ instead. Indeed, ignoring the unavoidable degree constraint, Theorem~\ref{main_thm} improves a clique subdivision of size in Bollob\'{a}s-Thomason~\cite{Bollobas1998} and Koml\'os-Szemer\'edi~\cite{Komlos1996} from $\sqrt{d(G)}$ to about $\sqrt{c(G)}$. We refer the readers to~\cite{Haslegrave2021} for more examples regarding the existence of long cycles on such a paradigm.
\end{rmk}

In the rest of this section, we illustrate three applications of Theorem~\ref{main_thm}, discuss an intimate connection between crux and the sublinear expander theory, and then mention some other topics and problems related to the crux.

\subsection{Characterisation of extremal graphs}
The first consequence of our main result is a structural characterisation of extremal graphs $G$ having the smallest possible clique subdivision size $\Theta(\sqrt{d(G)})$, showing that the \emph{only obstruction} to get a larger than usual clique subdivision is a small crux. In other words, if $c(G)=\omega(d)$, then one can embed a $K_t$-subdivision with $t = \omega(\sqrt{d})$. 
Theorem~\ref{main_thm} does not directly imply the following result, but along the way of proving it, we obtain this. 

\begin{theorem}\label{thm_omega_sqrt_d}
Given a graph $G$ with average degree $d$, if the largest clique subdivision has order $\Theta(\sqrt{d})$, then its crux size is linear in $d$, i.e. $c(G)=O(d)$.
\end{theorem}

Theorem~\ref{thm_omega_sqrt_d} implies that the extremal graphs are \emph{essentially disjoint unions of dense small graphs whose crux size is linear in their average degrees}. This can be viewed as an analogous result of Myers~\cite{Myers2002} who studied the extremal graphs for embedding clique minors.

We would like to mention a related question of Wood from the Barbados workshop 2020.

\begin{question}[Wood]\label{ques-Ktt-subdivision}
	For given $t\in \mathbb{N}$, does there exist $g(t)= o(t^2)$, such that every graph with the average degree at least $g(t)$ contains a subdivision of $K_{t,t}$? 
	\end{question}
	
This question asks whether the structure of $H$ is related to the average degree on $G$ guaranteeing an $H$-subdivision in it. 
However, dense random graphs provide a negative answer to this question. 
In fact, regardless of the structure of $H$, the quadratic bound cannot be improved for $H$-subdivision embedding as long as $H$ is a dense graph.

\begin{prop}\label{claim_t2_lowerbound}
For any $0<c\leq 1$, there exists $t_0$ such that the following holds for every $t>t_0$. There exists a graph $G$ with $d(G)\geq \frac{c^2 t^2}{1000}$ such that for any $t$-vertex graphs $H$ with $e(H) \geq c\binom{t}{2}$, $G$ does not contain an $H$-subdivision.
\end{prop}
We need the Chernoff bound, which can be found, for example, in \cite{Alon2016}.
\begin{theorem*}[The Chernoff Bound]\label{lem_chernoff}
	Let $X_1, X_2, \cdots, X_n$ be i.i.d. Bernoulli random variables and let $X= \sum_{i=1}^n X_i$ and $\mu=\mathbb{E}[X]$. Then for every $\varepsilon \in (0, 1)$, $\mathbb{P}(|X-\mu| \geq \varepsilon \mu ) \leq 2\exp\left(-\frac{\varepsilon^2 \mu}{3}\right).$
\end{theorem*}
\begin{proof}[Proof of Proposition~\ref{claim_t2_lowerbound}]
As any such a $t$-vertex graph $H$ contains a bipartite subgraph with at least $\frac{ct^2}{5}$ edges, it suffices to show that a binomial random graph $G=G(\frac{c t^2}{100}, \frac{c}{4})$, with high probability, does not contain an $H$-subdivision for every bipartite $H$ with $e(H) \geq \frac{ct^2}{5}$ while having average degree $d(G)\geq \frac{c^2}{1000}t^2$.

Indeed, suppose $X, Y \subseteq V(G)$ are the sets of core vertices (which are the vertices of degree at least three in the subdivision) corresponding to the bipartition of $H$. 
Then as $|X|+|Y|=t$, the expected number of edges between $X$ and $Y$ is at most $\frac{c t^2}{16}$.
Then the Chernoff bound together with the union bound yields that with probability at least $1-\binom{ct^2/100}{t}2^t e^{-ct^2/10^5} = 1-o(1)$, the number of edges between any two sets $X$ and $Y$ with $|X|+|Y|=t$ is at most $\frac{ct^2}{12}$. 
However, in such a case, more than $\frac{ct^2}{10} $ pairs $(x,y)\in (X\times Y)$ must be connected by a path of length at least two in the subdivision of $H$ for any bipartite $H$ with $e(H) \geq \frac{ct^2}{5}$. This yields that the $H$-subdivision has at least $\frac{ct^2}{10}$ vertices, which is more than the number of vertices available. On the other hand, a simple Chernoff bound again implies that the average degree of $G$ is at least $\frac{c^2}{1000}t^2$ with probability $1-o(1)$.
\end{proof}

Thus, a sub-quadratic bound is not sufficient to obtain a copy of an $H$-subdivision for any dense graph $H$. This example shows that when we seek for an $H$-subdivision within $G$ for any dense choices of $H$, the important factor is again the `essential order' of $G$, rather than the structures of $H$.

\subsection{Graphs without a fixed bipartite graph}
	The next application provides the largest clique subdivision size, which is optimal up to a polylog-factor, in a graph without a copy of bipartite graph $H$. This generalises the result of Liu and Montgomery~\cite{Liu2017} on $K_{s,t}$-free graphs. We would like to remark that the proof of Liu and Montgomery makes heavy use of the structure of the forbidden graph $K_{s,t}$, hence their argument does not extend to general $H$-free graphs. Below, we write $x= \widetilde{\Omega}(y)$ if there exists positive constants $a,b$ such that $x \geq a y \log^{-b}y$, $x= \widetilde{O}(y)$ if there exists positive constants $a,b$ such that $x \leq a y \log^{b}y$ and $x=\widetilde{\Theta}(y)$ if $x= \widetilde{\Omega}(y)$ and $x= \widetilde{O}(y)$.
	\begin{cor}\label{cor1}
		Let $H$ be a bipartite graph with extremal number $\ex(n, H) = O(n^{1+\tau})$ for some $0<\tau<1$ and let $G$ be an $H$-free graph with average degree $d$. Then $G$ contains a $K_t$-subdivision where 
		$$ 
		t = \left\{\begin{array}{ll} 
 \widetilde{\Omega}(d^{\frac{1}{2\tau}})  & \text{ if } \tau > 1/2 \\
 \widetilde{\Omega}(d) & \text{ if } \tau \le 1/2.
 \end{array}\right.	$$
	\end{cor}
	\begin{proof}
		Let $\alpha = 1/100$, and $F$ be a smallest $\alpha$-crux of $G$ of order $c(G)$.
		As $F$ is $H$-free, $e(F) = O(|F|^{1+\tau})$, hence  
		$$c(G) = |F| = \Omega(d(F)^{1/\tau})= \Omega(d(G)^{1/\tau}).$$
		Therefore, Theorem \ref{main_thm} implies that $G$ has a clique subdivision of size 
$\tilde{\Omega}(\min\{d, c(G)^{\frac{1}{2}} \})= \tilde{\Omega}(\min\{d, d^{\frac{1}{2\tau}} \}).$
	\end{proof}
The bound above is best possible up to the polylogarithmic factor if $\ex(n, H) = \Theta(n^{1+\tau})$.
To see this, let $G$ be an $n$-vertex bipartite $H$-free graph $G$ with $\Theta(n^{1+\tau})$ edges.
If $G$ has a $K_t$-subdivision, then at least $t/2$ core vertices are in the same part of a bipartition of $G$. For any two of them, a path connecting them uses at least one vertex of the other part of $G$ so ${t/2 \choose 2} \leq n$ and therefore $t = O(\sqrt{n}) = O(d^{1/2\tau})$.

We would like to again emphasise that the only effect of forbidding a fixed bipartite graph is to relax the space constraints. For example, by Theorem~\ref{main_thm}, the hypercube $Q^d$ has a clique subdivision of order $\widetilde{\Omega}(d)$ even though $Q^d$ has many 4-cycles.

\subsection{Dichotomy on Erd\Ho s-R\'enyi random graphs}
	 The last application studies subdivisions in Erd\Ho s-R\'enyi random graphs. While the size of the largest $K_t$-minor in a random graph is widely studied (for example \cite{Bollobas1980, Fountoulakis2008, Krivelevich2009}), not many results are known for clique subdivision. When $p \in (0, 1)$ is a constant, Bollob\'as and Catlin \cite{Bollobas1981} proved in 1981 that the largest clique subdivision of $G(n, p)$ is $(\sqrt{2/(1-p)}+o(1))\sqrt{n}$ with high probability (w.h.p.). 
     When $p=o(1/\sqrt{n})$ and $p>1/n$, Ajtai,  Koml\'{o}s and Szemer\'{e}di~\cite{Ajtai1979} proved that the size of largest clique subdivision of $G=G(n, p)$ is in $[(1-\varepsilon)\Delta(G), \Delta(G)]$.
	 
	We fill in the missing gap when $1/\sqrt{n}\le p\ll 1$. 

\begin{cor}\label{cor_random_graph}
	Suppose $1/n<p<1-\Omega(1)$.
    Then w.h.p., the largest $t$ that $G=G(n, p)$ has a $K_t$-subdivision is given by 
	\begin{align*}
	 t=\widetilde{\Theta}(\min\{np, \sqrt{n/(1-p)}\}).
	\end{align*}
\end{cor}

Corollary~\ref{cor_random_graph} implies an interesting dichotomy on clique subdivision size in $G(n,p)$ above and below the density $1/\sqrt{n}$: when $p =\omega(n^{-1/2})$, then it is limited solely by the space constraints, while when $p = o(n^{-1/2})$, the degree constraint is the bottleneck.


\begin{proof}[Proof of Corollary~\ref{cor_random_graph}]
 By Chernoff's bound, $d(G) \in (\frac{np}{2}, \frac{3np}{2})$ with high probability. Therefore, the upper bound can be obtained from the theorem of Bollob\'as and Catlin~\cite{Bollobas1981}, and the degree condition.
 
 For the lower bound, the case when $p=o(1/\sqrt{n})$ is covered by the result of Ajtai,  Koml\'{o}s and Szemer\'{e}di~\cite{Ajtai1979}. It is then sufficient to show that $c(G) = \Omega(n)$ with probability $1-o(1)$. Let $S$ be a set of vertices of size $m$ for some $np/200 \leq m \leq n/1000$. Then the expected number of edges in $G[S]$ is $p{m \choose 2} \leq nmp/2000$. Thus, by Chernoff's bound, we have 
 \begin{align*}
  \mathbb{P}(e(G[S]) \geq nmp/400) \leq \exp(-nmp/10^6)
 \end{align*}
 So, by the union bound, the probability that there exists a set $S$ of size $m$ such that $d(G[S]) \geq np/200$ is at most ${n \choose m} \exp(-nmp/10^6) =o(1/n)$ as $p = \omega(\frac{\log n}{n})$. By the union bound for all $m$, with probability $1-o(1)$, for every set of vertices $S \subseteq V(G)$ of size in between $np/200$ and $n/1000$, $d(G[S]) <d(G)/100$. 
 If $|S|<np/200$, then $d(G[S]) \leq |S| < d(G)/100$ so we have $c(G) \geq n/1000$ with probability $1-o(1)$.
\end{proof}

\subsection{Crux and sublinear expanders}
In this subsection, we discuss a deep connection between the notion of crux and the sublinear expander theory. 
That is, the crux size of a graph provides a lower bound on the size of the sublinear expander subgraphs it contains. Apart from capturing the essential order of a graph, this connection is the second main motivation for introducing the notion of the crux.

We defer the formal definition of sublinear expanders and their properties to Section~\ref{sec: sublinear-expander}. Roughly speaking a graph is a sublinear expander if most of its vertex subsets expand almost linearly. The starting point of the sublinear expander theory is the following result proved by Koml\'os and Szemer\'{e}di, see Theorem~\ref{thm_robust_expander} for the formal statement.

\begin{theorem*}[Existence of sublinear expander subgraph]
    Every graph $G$ contains a sublinear expander subgraph $H$ retaining almost the same average degree: $d(H)\ge (1-o(1))d(G)$.
\end{theorem*}

The main drawback of this result is that the expander subgraph could be much smaller than the original graph: consider e.g. again the disjoint union of many small dense graphs. This has caused many technical difficulties in previous work using a sublinear expander. Denote by $\sigma(G)$ the order of largest sublinear expander subgraphs in $G$. It is thus important to get a hold on $\sigma(G)$. As $\alpha$-crux is the minimum order among all subgraphs retaining $\alpha$-fraction of the average degree, we infer the following lower bound on $\sigma(G)$:
$$\sigma(G)\ge c_{1-o(1)}(G).$$
This lower bound was utilised e.g. in~\cite{Haslegrave2021} in studying the circumference of subgraphs of hypercubes and its bond percolations at the supercritical regime.

This connection of crux with sublinear expander will likely lead to further interesting results; it plays also an important role in our proof as most of the time we obtain bounds on clique subdivisions in terms of $\sigma(G)$.

\subsection{Crux and small sets expansions}\label{sec:crux-SSE}
There is yet another intriguing connection between the crux and expansions of small sets. Consider throughout this subsection an $n$-vertex $d$-regular graph $G$, and define the \emph{expansion profile of $G$ at $\delta>0$} to be
 $$\varphi_{\delta}(G)=\min_{|S|\le \delta n}\frac{e(S,S^c)}{d|S|}.$$
The well-known \emph{Small Set Expansion Problem} in Complexity Theory is to decide if it is NP-hard to distinguish between a graph with a small subset of vertices whose edge expansion is almost zero, i.e.~$\varphi_{\delta}(G)\le\eps$, and one in which all small subsets of vertices have expansion almost one, i.e.~$\varphi_{\delta}(G)\ge 1-\eps$. We refer the readers to the work of Arora, Barak and Steurer~\cite{Arora2015} on their breakthrough subexponential algorithm for this problem.

To see the connection to the crux, recall that by the minimality in the definition of an $\eps$-crux, for any set $S$ with $|S|< c_{\eps}(G)$, $d(G[S])<\eps d$, and so $\frac{e(S,S^c)}{d|S|} = \frac{d|S| - e(G[S])}{d|S|}> 1-\eps$ by the $d$-regularity. We then infer that
$$\varphi_{\frac{c_{\eps}(G)-1}{n}}(G)> 1-\eps.$$
Thus, if $G$ has linear size $\eps$-crux $c_{\eps}(G)$, then it is a small set expander.

While the Small Set Expander Problem remains open, we shall see that it is NP-hard to decide the crux size of a graph.
We note that deciding whether a given graph $G$ has a clique of size $k$ for a given $k$, called $\textsc{clique}$, is one of the Karp's 21 NP-complete problems~\cite{Karp1972}
\begin{prop}
		For fixed $0<\eps<1$, deciding whether $c_{\eps}(G) \leq k$ or not for a given graph $G$ and an integer $k$ is NP-hard.
\end{prop}
\begin{proof}
		We will give a reduction from $\textsc{clique}$. More precisely, for a given graph $G$ and an integer $k \geq 3$, we construct a graph $G'$ in polynomial-time such that $G$ has a $k$-clique if and only if $c_\eps(G') \leq k$. Hence, the problem of finding a $k$-clique in $G$ is reducible in polynomial time to the problem of deciding $c_{\eps}(G)\leq k$, implying that this problem is NP-hard.
		
		First observe that for a graph $G$ and $d_0 \in (k-1-\frac{1}{k}, k-1]$, $\min\{|H| \mid H \subseteq G, d(H) \geq d_0\} $ cannot be strictly smaller than $k$. Thus, $G$ has a $k$-clique if and only if $\min\{|H| \mid H \subseteq G, d(H) \geq d_0\} \leq k$.
		
		Now, let $(G, k)$ be an input instance of $\textsc{clique}$ and let $d=d(G)$, $n=|G|$.
		By the previous observation, it is sufficient to construct a graph $G'$ satisfying $\omega(G)=\omega(G')$ and $d(G') \in \frac{1}{\eps} \cdot(k-1-\frac{1}{k}, k-1]$ in polynomial-time. 
		Assume that $3 \leq k \leq n$, $G$ is nonempty, and $d \notin \frac{1}{\eps} \cdot (k-1-1/k, k-1]$.

		Case 1: If $d > \frac{k-1}{\eps}$, let $G'$ be a disjoint union of $m_1=\lceil \frac{kd}{n} \rceil$ copies of $G$ and $m_2$ isolated vertices where $m_2$ is the smallest integer such that $\frac{m_1nd}{m_1n+m_2}  \leq \frac{k-1}{\eps}$.
		Then $m_1 \leq O(n^2)$ and $0< m_2 \leq m_1nd \leq O(n^4)$.
		Also, we have $$m_1nd\left(\frac{1}{m_1n+m_2-1}-\frac{1}{m_1n+m_2}\right) = \frac{m_1nd}{(m_1n+m_2-1)(m_1n+m_2)} \leq 
		\frac{d}{m_1n} \leq \frac{1}{k}.$$
		Thus, $ \frac{k-1-1/k}{\eps}  < \frac{k-1}{\eps}-\frac{1}{k} \leq \frac{m_1nd}{m_1n+m_2}$.
		As $d(G')=\frac{m_1nd}{m_1n+m_2}$, we have $d(G') \in \frac{1}{\eps}\cdot (k-1-\frac{1}{k}, k-1]$.
		Thus, $c_{\eps}(G') \leq k$ if and only if $G$ has a $k$-clique and clearly $\omega(G) = \omega(G')$.

		Case 2: If $d \leq \frac{k-1-1/k}{\eps}$, let $m$ be the smallest integer such that $d+m \geq \frac{k-1}{\eps}$.
		Note that $m = O(n)$.
		Construct a graph $G''$ as follows. Let $G_1, G_2, \dots, G_{2m}$ be vertex-disjoint copies of $G$.
		For each $1 \leq i \leq m$ and $m+1 \leq j \leq 2m$, add a complete matching between $G_i$ and $G_j$.
		Then we have $d(G'')=d+m \geq \frac{k-1}{\eps}$.
		Also, every clique of $G''$ of size at least $3$ should be contained in one copy of $G$ so $\omega(G) = \omega(G'')$ unless $G$ is empty.
		If $d(G'') = \frac{k-1}{\eps}$, then we can take $G'=G''$.
		Otherwise, follow the process of Case 1 with $G''$ instead of $G$ to obtain desired $G'$. Hence, this shows that determining if $c_{\eps}(G')\leq k$ is as difficult as determining if $\omega(G)=k$.
	\end{proof}

	\medskip
	
	\noindent\textbf{Organisation.} The rest of the paper is organised as follows. In Section~\ref{sec: sublinear-expander} we introduce sublinear expanders. Then in Section~\ref{sec_proof}, we give an outline of the proof of Theorem~\ref{main_thm}, packaging the main steps into Lemmas~\ref{lem_dense}--\ref{lem_nc_expander}; and present the proof sketches for these main lemmas and derive Theorem~\ref{thm_omega_sqrt_d}. In Section \ref{sec_lemmas}, we collect technical lemmas. The proofs of the main lemmas are given in Sections \ref{sec_ed_expander}, \ref{sec_d2_expander}, and \ref{sec_nc_expander}. Concluding remarks are given in Section \ref{sec_conclusion}.

\section{Sublinear expanders}\label{sec: sublinear-expander}
A main tool we use in this paper is the sublinear expander notion introduced by Koml\'os and Szemer\'{e}di \cite{Komlos1994, Komlos1996}. Let $N^i_G(X)$ be the set of vertices which are distance exactly $i$ from $X$. In particular, $N^0_G(X)=X$. We write $N_G(X)$ to denote $N^1_G(X)$ and we write $B^i_G(X) = \bigcup_{j\leq i}N^i_G(X)$. For given $\eps, k$, we define $\rho(x)=\rho(x, \eps, k)$ as 
	\begin{align*}
	\rho(x)= 
	\begin{cases}
	0 & \text{if } x<\frac{k}{5} \\
	\frac{\eps}{\log^2(15x/k)} & \text{if } x \geq \frac{k}{5}
	\end{cases}
	\end{align*}
Note that $\rho(x)$ is a decreasing function and $x \rho(x)$ is an increasing function for $x \geq \frac{k}{2}$. 
Koml\'os and Szemer\'edi introduced the notion of $(\eps,k)$-expander, which is a graph in which every set of appropriate size has not too small external neighbourhood. 
Haslegrave, Kim and Liu \cite{Haslegrave2021a} slightly generalised this notion to a robust version. 
Roughly speaking, a graph $G$ is a robust-expander if every set of appropriate size has not too small external neighbourhood even after deleting a small number of vertices and edges.
For an edge set $F\subseteq E(G)$, we write $G\setminus F$ to denote the graph with the vertex set $V(G)$ and the edge set $E(G)\setminus F$.
	
	\begin{defn}[\cite{Haslegrave2021a}]
		For $\eps>0, k>0$, a graph $G$ is \emph{$(\eps, k)$-robust-expander} if for every subset $X \subseteq V(G)$ of size $\frac{k}{2} \leq |X| \leq \frac{|V(G)|}{2}$ and an edge set $F \subseteq E(G)$ with $|F| \leq d(G) \rho(|X|)|X|$, we have $|N_{G - F}(X)| \geq \rho(|X|)|X|$.
	\end{defn}

We note that a different notion of ``robust-expander'' was also used by K\"{u}hn, Osthus and Treglown~\cite{Kuhn2010} to define a different type of expansion that the number of vertices that has `many' neighbors in $X$ is large.  

This notion of the sublinear expander is very useful in the following two aspects. 
	\begin{itemize}
		\item Every graph contains a robust-expander subgraph with almost the same average degree.
		\item This  provides a short connection between any two large sets while avoiding a relatively small set of vertices and edges.
	\end{itemize}
These two aspects are captured in the following two results. For a vertex set $W$, we write $G-W$ to denote the induced subgraph $G[V(G)\setminus W]$.

	\begin{theorem}\label{thm_robust_expander}{\cite{Haslegrave2021a,Komlos1994,Komlos1996}}
		If $0<\eps<\frac{1}{8 \log 3}$ and $k\geq 1$ then every graph $G$ contains an $(\eps, k)$-robust-expander $H$ with $d(H) \geq d(G)/2$ and $\delta(H) \geq d(H)/2$ as a subgraph.
	\end{theorem}

	\begin{lemma}\label{short_path}\cite{Haslegrave2021a,Komlos1996}
		Let $G$ be an $n$-vertex $(\eps, k)$-robust-expander.
		Then for any two vertex sets $X_1, X_2$ of size at least $ x \geq \frac{k}{2}$, and a vertex set $W$ of size at most $ \frac{x \rho(x)}{4}$, 
		there exists a path between $X_1$ and $X_2$ in $G-W$ of length at most $ \frac{2}{\eps}\log^3(\frac{15n}{k})$.
	\end{lemma}

    We defer more technical expansion properties we need for sublinear expanders to Section~\ref{sec:technical-expansion}.

	\section{Outlines of the main steps}\label{sec_proof} 
	Assume that $G$ and $t$ are as in Theorem~\ref{main_thm}.
	To construct a desired $K_t$-subdivision, we have to overcome both the degree constraint and the space constraint. As a first step, we have to find many vertices of degree at least $t$ (the degree constraint) and connect them with short paths(the space constraint). 
	
	Using Theorem~\ref{thm_robust_expander}, we can assume that our graph is an expander. In an expander, the edge distributions are somewhat uniform, so we can use the property of expander to overcome the degree constraint as long as $t\le \frac{d}{(\log d)^{O(1)}}$. If $t$ is even closer to $d$, we might have to `recycle' some used vertices (see Section~\ref{subsec_lemmas} for more explanations).
	
	To see what the space constraint means, note that a $K_t$-subdivision has $\Omega(t^2)$ paths and at most $n$ vertices in total.
	Hence the average length of the paths in the subdivision must be at most $O(\frac{n}{t^2})$. 
	We plan to use sublinear expander properties to build such paths, and the paths we obtain from Lemma~\ref{short_path} are only guaranteed to have length $O(\log^{3}n)$. Hence this is feasible only when $t=\frac{\sqrt{n}}{(\log n)^{O(1)}}$.
	However, we aim for $t= \sqrt{\frac{n}{\log n}}/(\log\log n)^{O(1)}$. For this, we need to be able to build paths of length $\log n(\log\log n)^{O(1)}$, which is much shorter than the ones guaranteed from the property of sublinear expanders. 
	
	Stronger expansion properties are needed to obtain shorter paths. Here is where the crux kicks into play as it provides good expansion for small sets as pointed out in Section~\ref{sec:crux-SSE}. That is, any set of size smaller than $c(G)$ must induce a subgraph of small average degree, hence sending many edges outside, yielding a large external neighbourhood as long as the graph has not too small minimum degree. This enhanced expansion can be used to construct shorter paths.
	
Thus, we want to use the sublinear expander property to obtain a good expansion for large sets and use the definition of crux to obtain a good expansion for small sets. 
Where then is the cutting point of these `large' and `small' sets?  
Note that the first expansion depends on the choice of $k$ on an $(\eps,k)$-expander, and the later expansion depends on the value of $c(G)$.
Hence, in order to combine these two expansions in a well-coordinated way, we need to take $(\eps,k)$-expander with several different values of $k$ in coordination with the value of $c(G)$ to cover different regimes to obtain Theorem~\ref{main_thm}. 

We will first introduce the following two lemmas which help us to put some assumptions on how big $c(G)$ can be. 
In what follows, if we claim a result holds if $0<a\ll b \ll c,d <1$, then it means there exist non-decreasing functions $f,g: (0,1]\rightarrow (0,1]$ such that the result holds when $a< f(b)$ and $b< g(c,d)$. We will not compute these functions explicitly.
\begin{lemma}\label{lem_dense}
	Suppose $0<\frac{1}{n} \ll \frac{1}{T} \ll \eps <\frac{1}{100}$ and $d \geq \exp(\log^{1/6} n)$.
	If an $n$-vertex $(\eps, \eps d)$-robust-expander $G$ has average degree at least $d/2$ and minimum degree $\delta(G) \geq \frac{d}{4}$, then $G$ contains a $K_t$-subdivision where $t= \min\{\frac{d}{(\log d)^{60}}, \frac{\sqrt{n}}{(\log (Tn/d))^{11}}\}$.
\end{lemma}

	\begin{lemma}\label{lem_sparse}
		Suppose $0<\frac{1}{n} \ll b \ll \eps < \frac{1}{100}$ and $d \leq \exp(\log^{1/6} n)$.
		If an $n$-vertex $(\eps, \eps d)$-expander $G$ have average degree at least $d/2$ and minimum degree $\delta(G) \geq \frac{d}{4}$, then $G$ contains a $K_{bd}$-subdivision.
	\end{lemma}

Note that Theorem~\ref{thm_omega_sqrt_d} follows immediately from Lemma~\ref{lem_dense} and Lemma~\ref{lem_sparse}. Lemma~\ref{lem_sparse} is a strengthening of Lemma~2.10 in~\cite{Liu2017}; we include its proof in the arXiv appendix.

We will use Theorem~\ref{thm_robust_expander} to take an $(\eps, \eps d)$-expander subgraph $G_1\subseteq G$.
If both $d$ and $\sqrt{n}$ are a polylog-factor larger than $\sqrt{c(G)}$, then the above two lemmas already yield the desired $K_t$-subdivision. 
If $\sqrt{c(G)} = \tilde{\Omega}(d)$,  then we will use Theorem~\ref{thm_robust_expander} to find an $(\eps, d^2)$-robust-expander $G_2 \subseteq G_1$. With this choice, the following lemma will provide a desired $K_t$-subdivision within $G_2$.

	\begin{lemma}\label{lem_d2_expander}
		Suppose $0<\frac{1}{n} \ll b \ll \eps \ll 1$ and $\exp(\log^{1/6} n) \leq d\leq\frac{\sqrt{n}}{(\log n)^{1000}}$.
		If an $n$-vertex $(\eps, d^2)$-robust-expander $G$ satisfies $d(G) \geq \frac{d}{4}$ and $\delta(G) \geq \frac{d}{8}$, then $G$ contains a $K_{t}$ subdivision with $t=b\cdot \frac{d}{(\log \log d)^4}$.
	\end{lemma}
	
	In the remaining case, $d$ is very close to $\sqrt{c(G)}$ up to a polylogarithmic factor. Then we find instead an $(\eps,c(G)/100)$-robust-expander $H\subseteq G_2$ and apply the following lemma. Here, $d$ being very close to $\sqrt{c(G)}$ provides a well-coordination between two expansions coming from the sublinear expander and the definition of the crux.
	
	\begin{lemma}\label{lem_nc_expander} 
	 Suppose $0<\frac{1}{n} \ll b \ll \eps \ll 1$.
	 Let $G$ be an $n$-vertex graph with $d(G)=d$ and $H \subseteq G$ be an $(\eps, \frac{c(G)}{100})$-robust-expander satisfying $\frac{d}{16} \leq \delta(H)\leq  d(H)$.
	 If $d \geq \frac{\sqrt{|H|}}{\log^{1000}|H|}$ and $\frac{1500|H|}{c(G)} \leq \log^{3000} |H|$, then $H$ contains a $K_t$-subdivision where $t=  \frac{b t'}{(\log\log t')^6} \enspace \text{ and }	 \enspace
t' = \min\left\{d, \sqrt{\frac{c(G)}{\log c(G) }}\right\}.$
	\end{lemma}

With these four lemmas, Theorem~\ref{main_thm} follows.
\begin{proof}[Proof of Theorem \ref{main_thm}]
Choose small numbers $b,\eps$ and large integer $n_0$ with $0< \frac{1}{n_0} \ll b, \frac{1}{T} \ll \eps \ll 1$ so that Lemma~\ref{lem_dense}--\ref{lem_nc_expander} hold for $n\geq n_0$.
By taking $\beta = \min\{\frac{1}{100n_0}, 10^{-10}\}$, we also assume that $d$ is at least $\max\{10^5, T\}$ and $c(G)$ is at least $n_0$. 
To simplify the notation, let $n_c=c(G)$.

Let $G$ be an $n$-vertex graph with $d(G)=d$.
By Theorem~\ref{thm_robust_expander} we have an $(\eps, \eps d)$-robust-expander $G_1 \subseteq G$ with $d(G_1) \geq \frac{d}{2}$, $\delta(G_1) \geq \frac{d}{4}$ and let $n_1:=|G_1|.$
If $d \leq \exp(\log^{1/6} n_1)$ then by Lemma~\ref{lem_sparse}, $G_1$ contains a $K_{\beta d}$-subdivision as $\beta \leq b$. As $\beta d \geq \beta t/(\log \log t)^6$, we obtain the desired subdivision in this case. Hence we may assume $d > \exp(\log^{1/6} n_1)$.
With this assumption, Lemma~\ref{lem_dense} implies that  $G_1$ contains a $K_s$-subdivision where $s=\min\Big\{\frac{d}{(\log d)^{60}}, \frac{\sqrt{n_1}}{(\log n_1)^{11}}\Big\}.$
If $s$ is bigger than $t \geq \beta t/(\log \log t)^6$, then we are done. Otherwise, we have either 
\begin{align}\label{eq: A or B}
	\sqrt{\frac{n_c}{\log n_c}} \geq \frac{d}{(\log d)^{60}} \enspace \text{ or }\enspace 
		\sqrt{\frac{n_c}{\log n_c}}\geq \frac{\sqrt{n_1}}{(\log n_1)^{11}}.
\end{align}

By Theorem~\ref{thm_robust_expander}, we have an $(\eps,  d^2)$-robust-expander $G_2 \subseteq G_1$ with $d(G_2) \geq \frac{d}{4}$, $\delta(G_2) \geq \frac{d}{8}$, and let $n_2:=|G_2|.$ If $d \leq \frac{\sqrt{n_2}}{(\log n_2)^{1000}}$, then Lemma~\ref{lem_d2_expander} implies that $G_2$ contains a $K_{bd/(\log \log d)^4}$-subdivision as desired.
Hence, we may assume that $d > \frac{\sqrt{n_2}}{(\log n_2)^{1000}}$.

Finally, by Theorem~\ref{thm_robust_expander} we have an $(\eps,  n_c/100)$-robust-expander $H \subseteq G_2$ with $d(H) \geq \frac{d}{8}, \enspace \delta(H) \geq \frac{d}{16}, \enspace \text{ and let } \enspace n_H:=|H|.$
Then we have 
\begin{align}\label{eq: d condition}
d > \frac{\sqrt{n_2}}{(\log n_2)^{1000}} \geq \frac{\sqrt{n_H}}{(\log n_H)^{1000}}.	
\end{align}

We claim that $\frac{n_H}{n_c} \leq (\log n_H)^{3000}$. Indeed,  
if the first case of \eqref{eq: A or B} holds, then $\frac{\sqrt{n_c}}{\sqrt{\log n_c}} \geq \frac{d}{(\log d)^{60}} \geq \frac{\sqrt{n_H}}{(\log n_H)^{1060}},$ implying $\frac{n_H}{n_c} \leq (\log n_H)^{3000}$.
If the second case of \eqref{eq: A or B} holds, then $\sqrt{n_c} \geq \frac{\sqrt{n_1}}{(\log n_1)^{11}} \geq \frac{\sqrt{n_H}}{(\log n_H)^{11}}$
as $n_H\leq n_1$, so $\frac{n_H}{n_c} \leq (\log n_H)^{22}$. 
This together with \eqref{eq: d condition} implies that we can apply Lemma~\ref{lem_nc_expander} to $H$, to obtain a $K_{\beta t/(\log \log t)^6}$-subdivision as desired.
	\end{proof}
\begin{proof}[Proof of Theorem~\ref{thm_omega_sqrt_d}]
	Choose small numbers $\eps, \frac{1}{T}, b$ and large integer $n_0$ with $0< \frac{1}{n_0} \ll \frac{1}{T}, b \ll \eps \ll 1$ so that Lemma~\ref{lem_dense} and \ref{lem_sparse} hold for $n\geq n_0$.

It suffices to show that the following holds for all $d,a>1$: if $d(G)=d$ and $c(G)\geq ad$, then $G$ contains a $K_{t}$-subdivision with $t= \min\{\frac{bd}{(\log d)^{60}},  \frac{\sqrt{a}\cdot\sqrt{d}}{\log^{11} (Ta)} \}$. Note that as $a$ and $d$ grows, both terms defining $t$ are $\omega(\sqrt{d})$.

	Let $G$ be an $n$-vertex graph with $d(G)=d$ and $c(G)\geq ad$.
	By Theorem~\ref{thm_robust_expander} we have an $(\eps, \eps d)$-robust-expander $G_1 \subseteq G$ with $d(G_1) \geq d/2$ and $\delta(G_1) \geq d/4$.
	By the definition, we have $|G_1| \geq c(G)$.
	If $d \leq \exp(\log^{1/6} |G_1|)$, then by Lemma~\ref{lem_sparse}, $G_1$ contains a $K_{bd}$-subdivision and $bd\geq t$.
	When $d \geq \exp(\log^{1/6} |G_1|)$, then by Lemma~\ref{lem_dense}, $G_1$ has a clique subdivision of size at least $\min\left\{\frac{d}{(\log d)^{60}}, \frac{\sqrt{c(G)}}{\left(\log(\frac{Tc(G)}{d})\right)^{11}}\right\} = \min\left\{\frac{d}{(\log d)^{60}},  \frac{\sqrt{a}\cdot\sqrt{d}}{\log^{11} (Ta)}\right\}\geq t.$
	This concludes the proof.
\end{proof}
	
	For proving Lemmas~\ref{lem_dense}, \ref{lem_d2_expander}, and \ref{lem_nc_expander}, we will use the following two concepts of units and webs. See Figure~\ref{fig:1} and Figure~\ref{fig:2}
 	
	\begin{defn}[Unit]
		Consider $h_1$ vertex-disjoint copies of stars $K_{1,h_2}$ with center vertices $v_1,\dots v_{h_1}$. Consider another vertex $v$ and add internally vertex-disjoint paths of length at most $h_3$ connecting $v$ and $v_i$ for each $i\in [h_1]$ which does not intersect with any copies of $K_{1,h_2}$ except the vertex $v_i$. We call the resulting graph an \emph{$(h_1,h_2,h_3)$-unit}.
		
		The vertex $v$ is called the \emph{center} of the unit and the non-leaf vertices of the unit are called \emph{interior vertices}. The \emph{interior} of the unit is the set of interior vertices of the unit. The set of leaf vertices is called the \emph{exterior} of the unit. We also call each copy of $K_{1, h_2}$ an \emph{exterior star} of the unit.
		For an exterior star of the unit, there is a unique path from the center to the exterior star.
		The union of the exterior star together with the path except the center is called a \emph{branch} of the unit.
		A collection of units are \emph{internally disjoint} if the interiors of units are pairwise disjoint.
	\end{defn}
	
	Note that an $(h_1,h_2,h_3)$-unit has at most $1+h_1h_3$ vertices in the interior and $h_1h_2$ vertices in its exterior.
	
	\begin{defn}[Web]
	 Consider $h_4$ many vertex-disjoint $(h_1, h_2, h_3)$-units with center vertices $v_1, \cdots, v_{h_4}$.
	 Add a new vertex $v$ which is not contained in any of the units and add internally disjoint paths of length at most $h_5$ connecting $v$ and $v_i$ for $i \in [h_4]$ which does not intersect any of units except at $v_i$.
	 We say the resulting graph is an \emph{$(h_4, h_5, h_1,h_2,h_3)$-web}.

	 The vertex $v$ is called the \emph{center} of the web and the set of vertices in the paths connecting $v$ and $v_i$'s is called the \emph{core} of the web.
	 The set of leaves in the web is called the \emph{exterior} of the web and the set of non-leaf vertices is called the \emph{interior} of the web.
	 For a copy of the unit inside the web, there is a unique path from the center to the unit.
	 	The union of the unit together with the path except the center is called a \emph{branch} of the web. 
	 A collection of webs are \emph{internally disjoint} if the interiors of webs are pairwise disjoint.
	\end{defn}
	Note that an $(h_4,h_5,h_1,h_2,h_3)$-web has at most $1+h_4h_5$ core vertices and at most $1+h_4h_5+h_1h_3h_4$ interior vertices. On the other hand, it has  $h_1h_2h_4$ vertices in its exterior.
	 
	\begin{figure}
		\centering
		\minipage{0.37\textwidth}
  		\includegraphics[width=\linewidth]{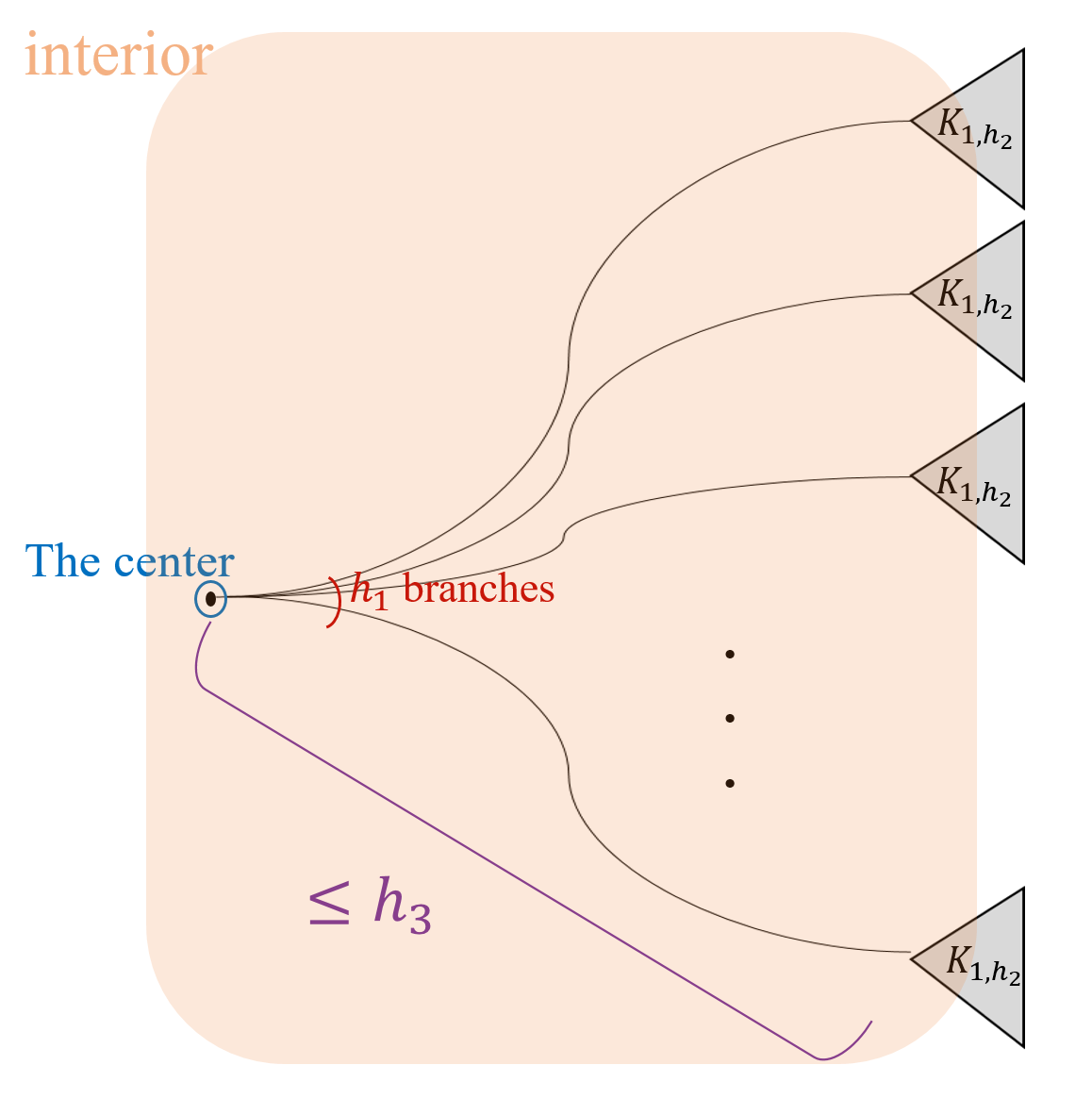}
  		  		\caption{$(h_1, h_2, h_3)$-unit}
  		\label{fig:1}
		\endminipage\hfill
		\minipage{0.55\textwidth}
  		\includegraphics[width=\linewidth]{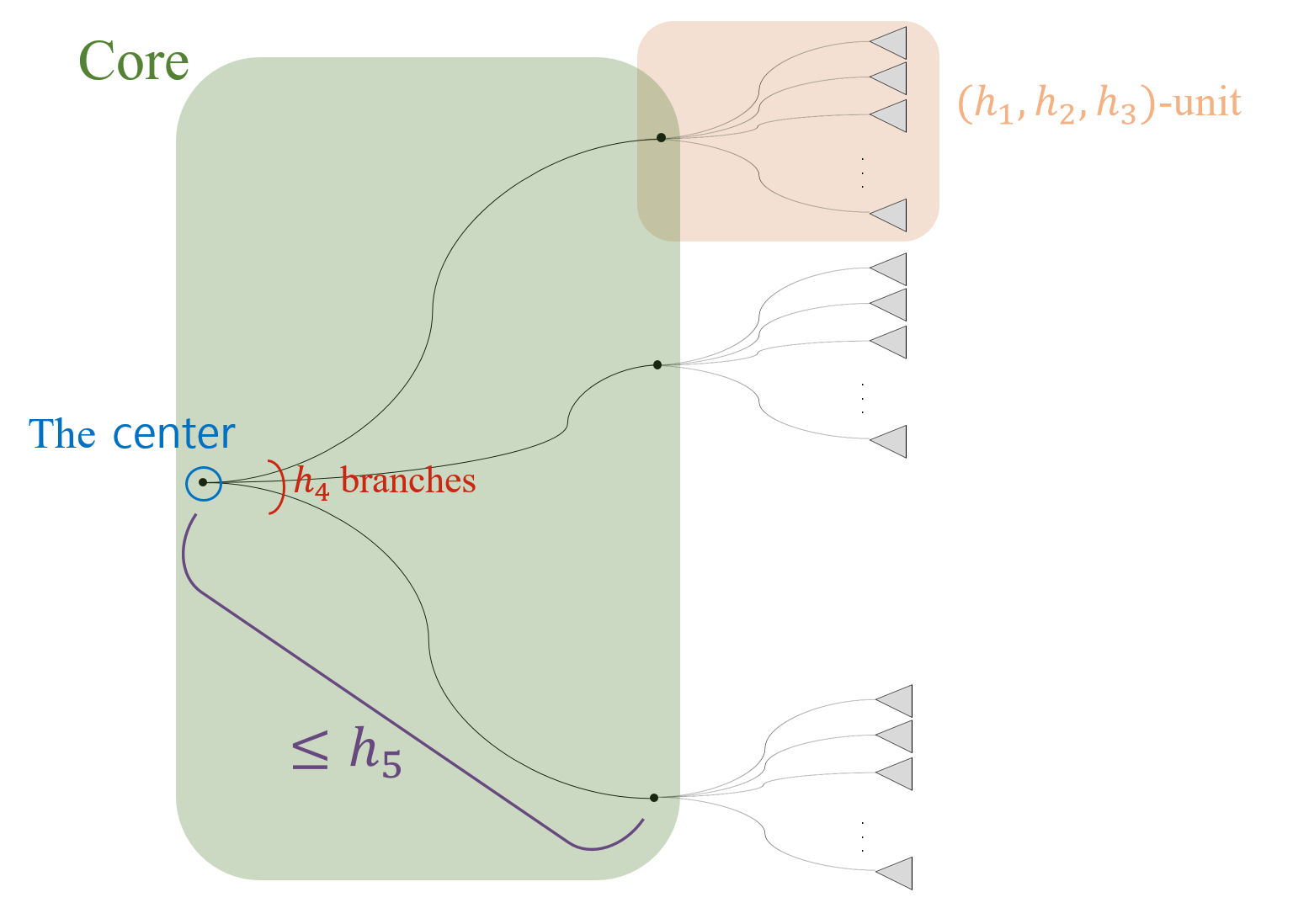}
  		  		\caption{$(h_4, h_5, h_1, h_2, h_3)$-web}
  			\label{fig:2}
		\endminipage
	\end{figure}

	\subsection{Proof Sketch of Lemmas}\label{subsec_lemmas}
	
\noindent	{\bf A proof sketch of Lemma~\ref{lem_dense}.} 
	Our plan is to find many internally disjoint $( h_1,h_2,h_3)$-units with appropriate $h_1, h_2$ and $h_3$. If we have $t$ such units with large enough exteriors, then we can use Lemma~\ref{short_path} to greedily connect two exteriors while avoiding the interior of other units and previously built paths. This will yield the desired clique subdivision. 
	
	In order to build a subdivision iteratively, we need $h_1$ to be at least the degree of the core vertices of the subdivision, $h_1 = \Omega(t)$.
	As we plan to use Lemma~\ref{short_path}, we will have $h_3 = (\log n)^{\Omega(1)}$.
	When we build a $K_t$-subdivision using $t$ units, we have to connect the exterior of two units while avoiding the interior of $t-2$ other units. Hence, the size $h_1h_2$ of the exterior of a unit must be logarithmically larger than the interior size $h_1h_3$ of a unit times $t-2$, hence we need $h_2 \geq t h_3 = t (\log n)^{\Omega(1)}$.
	
	However, building $t$ such $(h_1,h_2,h_3)$-units with disjoint interiors turns out to be challenging. We will find many stars $K_{1,h_2}$ disjoint from already-built units using the definition of the robust expander (see Lemmas~\ref{lem_avg_deg}~and~\ref{disj_stars}). Let $W$ be the set of interior vertices of $t-1$ already built internally disjoint units.  Then $W$ may contain up to $th_1h_3 = t^2 (\log n)^{\Omega(1)}  = \omega(h_2)$ vertices. 
	Hence, connecting two stars $K_{1,h_2}$ while avoiding $W$ is not what Lemma~\ref{short_path} can offer.
	To avoid this issue, one natural way is to collect $q=t(\log n)^{\Omega(1)}$ many stars $S_1,\dots, S_q$ and $q'=\Omega(t)$ many stars $R_1,\dots, R_{q'}$ and build many vertex-disjoint paths between $S_i$'s and $R_j$'s.
	Collectively, $\bigcup_{i\in [q]} S_i$ and $\bigcup_{j\in [q']}R_j$ are both large, so we can find many paths between them and this yields many paths from one $S_i$ to $\bigcup_{j\in [q']}R_j$. However, we need to ensure that those paths have the other ends in all different $R_j$'s. 
	
	In order to ensure this, we do not directly find paths between those stars.
	Instead, we prove that at least one star $S_i$ have large $r$-balls $B^r_{G}(S_i)$ around it even after avoiding $W$ and previously built paths between $S_i$'s and $\bigcup_{j\in [q']}R_j$ (see Corollary~\ref{cor_single_large_set}).
	As this $r$-ball alone is large enough, we can use Lemma~\ref{short_path} to connect $S_i$ with $\bigcup_{j\in K} R_j$ while avoiding already used vertices, where $K$ is the collection of $j$ where $R_j$ is not used by any previously built paths from $S_i$. This will add one more path while ensuring all paths from $S_i$ has their other ends in different $R_j$'s.
	Repeating this will produce more paths and we obtain a desired unit at the end (see Claim~\ref{disj_units}). \newline

\noindent {\bf Proof sketches of Lemmas~\ref{lem_d2_expander}~and~\ref{lem_nc_expander}} 
	For Lemma \ref{lem_d2_expander} and Lemma \ref{lem_nc_expander}, we will construct many internally disjoint $(h_4, h_5, h_1, h_2, h_3)$-webs with appropriate parameters instead of units. 
	If we have $t$ many internally disjoint webs with large exteriors, we can apply Lemma~\ref{short_path} to greedily connect them while avoiding all the previous paths and cores of webs.

	In this greedy connection, one difference from the proof of Lemma~\ref{lem_dense} is the following: when we build paths to connect the centers of the internally disjoint webs, we only avoid using the vertices in the cores of webs. Each path may intersect with the interior of other webs. Indeed, the interiors of the webs are too large to avoid completely. However, as each unit in the web has many branches, these units can still provide the desired connection even if some vertices are already used by previously built paths. 
	So, we delete the units only when too many vertices were used, and we delete the webs only when too many units in the web have been deleted.
	In this way, we are able to show that we only delete $o(t)$ webs in this process and  get a desired $K_{t}$-subdivision.

Note that Lemma~\ref{lem_d2_expander} deals with the case where $t$ is very close to $d$
 and Lemma~\ref{lem_nc_expander} deals with the case where $t$ is very close to $\sqrt{\frac{|H|}{\log |H|}}$.
So, the degree constraint is more serious obstacle in Lemma~\ref{lem_d2_expander}, while the space constraint is the main problem in Lemma~\ref{lem_nc_expander}.

	In order to build $(h_4,h_5,h_1,h_2,h_3)$-webs, we will first find many disjoint stars $K_{1,h_2}$ and build many disjoint $(h_1,h_2,h_3)$-units from the stars, and then we take more stars $K_{1,h_4}$ and build $(h_4,h_5,h_1,h_2,h_3)$-webs by finding short paths from the stars $K_{1,h_4}$ to the $(h_1,h_2,h_3)$-units. Here, we want $h_4=\Omega(t)$ as the center of webs will correspond to the core vertices of the final subdivision. Unlike the case of Lemma~\ref{lem_dense}, $h_1$ does not have to be too large.
	
	In Lemma~\ref{lem_nc_expander}, the degree constraint is less severe, so we can use the definition of robust expander (see Lemma~\ref{lem_avg_deg}) to find many stars $K_{1,h_2}$ and $K_{1,h_4}$ with $h_2, h_4= \frac{d}{(\log n)^{O(1)}}$. In contrast, in Lemma~\ref{lem_d2_expander}, we have strong degree constraint as $t = \frac{d}{(\log\log d)^{O(1)}}$ is very close to $d$. This poses a problem as we want $h_4 = \Omega(t)$, but  the definition of robust expander does not yield many disjoint such large stars $K_{1,h_4}$.
	To resolve this problem, we separate proof into two cases. The first case is when the edges of the graph are uniformly distributed so that $d(G-W) = \Omega(d)$ for every small $W$. In this case, as $d=\Omega(h_4)$, by repeatedly taking a copy of $K_{1,h_4}$, we can easily obtain the desired collection of copies of $K_{1,h_4}$ and use them to find desired webs (see Claim~\ref{claim_uniform_webs}). 
	The second case is when there exists a small set $W$ satisfying $d(G-W) = o(d)$. In such a case, almost all edges are incident with $W$, and the minimum degree condition ensures that $W$ has size at least $\Omega(t)$. 
	Then, by taking stars with center vertices within $W$, we are able to find $\Omega(t)$ many vertex-disjoint stars $K_{1,h_4(\log n)^{\Omega(1)}}$ (see Lemma~\ref{heavy_q}).

    From the obtained vertex disjoint stars $K_{1,h_2}$, we want to build many disjoint $(h_1, h_2, h_3)$-units.
	This is easy in the proof of Lemma~\ref{lem_d2_expander} as the space constraint is not severe. However, for Lemma~\ref{lem_nc_expander}, we only have very tight space to find paths. As we saw before at the beginning of Section~\ref{sec_proof}, paths of length $(\log n)^{\Omega(1)}$ is too long for constructing $\Omega(t)$ internally disjoint webs, so $h_3$ must be $\log n(\log \log n)^{O(1)}$. 
	For this, we can take $h_1 = \log n(\log\log n)^{\Omega(1)}$ and $h_2= \frac{t}{(\log\log n)^{O(1)}}$. 
	In order to obtain paths of length at most $h_3$, we have to overcome the issue that small set expands poorly. For this, we will take many small sets. As they together form a big set, the union enjoys good expansion. From this fact, we want to pick some of them with good expansion. Moreover, we also want some flexibility to control which $S_i$ and $R_j$ we connect so that we can avoid connecting already connected parts again. Overall heuristics of the proof is similar to what we explained in the above proof sketch of Lemma~\ref{lem_dense}, but we need some more technical elements to overcome these issues (see Lemma~\ref{lem_units} and Claim~\ref{cl: units in c100 expander}).
	
	Lastly, we want to build $(h_4,h_5,h_1,h_2,h_3)$-webs from many given $(h_1,h_2,h_3)$-units and many given copies of $K_{1,h_4}$.
	Again this step is easier in Lemma~\ref{lem_d2_expander} as the space constraint is not severe. For Lemma~\ref{lem_nc_expander}, we want to find short paths of length at most $h_5= \log n (\log\log n)^{O(1)}$ between the stars and units. For this, we need to coordinate two different expansions, one from the definition of crux and the other from the sublinear expander theory. Again, using these two expansions in a well-coordinated way, we are able to prove that the process similar to the one in the previous paragraph can be carried out (see Lemma~\ref{lem_webs} and Claim~\ref{cl: webs in c100 expander}).
	 

\section{Technical lemmas}\label{sec_lemmas}
 In this section, we collect some useful lemmas. We write $[n]=\{1,\dots, n\}$. We omit floors and ceilings and treat numbers as integers whenever it does not affect the argument. 

In proving Lemmas~\ref{lem_dense}--\ref{lem_nc_expander}, we want to build units or webs. For this, we first need to be able to find many disjoint stars, which will be the building blocks for our units or webs. The following lemma will be useful for this. It asserts that in an $(\eps, k)$-robust-expander, even if we delete a set of vertices (which will be the vertices in the previously-taken stars or units or webs), the remaining graph still has a large average degree.

	\begin{lemma}\label{lem_avg_deg}
	Suppose $0<\frac{1}{n} \ll \eps <1$ and $k< \frac{n}{10}$.
		Let $G$ be an $n$-vertex $(\eps, k)$-robust-expander. Then 
		for every $W\subseteq V(G)$ with $|W|\leq \frac{1}{20} \rho(n, \eps, k) \cdot n$, we have $d(G- W) \geq \frac{1}{20} \rho(n, \eps, k) \cdot d$. 
	\end{lemma}
	\begin{proof}
		Let $H=G - W$. We claim that $H$ contains fewer than $n/10$ vertices $v$ with $d_{H}(v)\leq \frac{1}{10} \rho(n, \eps, k)\cdot d$.  
		Suppose that this is not the case. Then we can take $A$ to be a set of exactly $\frac{n}{10} \geq k$ such vertices. Let $F$ be the set of edges of $H$ incident to a vertex in $A$, then $|F| \leq  \frac{1}{10} \rho(n, \eps, k)d \cdot |A|.$ 
  
  Observe that $N_{G\setminus F}(A)\subseteq W$ as $A$ is an independent set in $H\setminus F = (G - W)\setminus F$. Thus we have 
		$N_{G\setminus F}(A) \leq |W| \leq \frac{1}{20} \rho(n,\eps,k) n,$
		contradicting the assumption that $G$ is an $(\eps,k)$-robust-expander.
Thus at least $\frac{9}{10}n -|W| \geq n/2$ vertices of $H$ have degree at least $\frac{1}{10} \rho(n, \eps, k)d$. Consequently,
		$d(G-W)\geq \frac{1}{n} \cdot \frac{n}{20}\cdot d\rho(n, \eps, k) \geq \frac{1}{20} \rho(n, \eps, k)\cdot d.$
	\end{proof}

In some situations, we need the following lemma to control the maximum degree of the graph. Its proof follows the strategy of Lemma~2.8 in~\cite{Liu2017}; we include it in the arXiv appendix.

	\begin{lemma}\label{lem_bounded_maxdeg}
	Suppose $0<1/d \ll b \ll \eps \ll 1$ and $\eps d \leq k \leq d^2 \log^9 n$. Let $G$ be an $n$-vertex $(\eps, k)$-robust-expander with $\delta(G) \geq \frac{d}{2}$. If $G$ does not contain a subdivision of $K_{bd}$, then $G$ contains an $(\eps/2, k)$-robust-expander subgraph $H$ with at least $\frac{1}{2}n$ vertices such that $\delta(H) \geq \frac{1}{4}d$ and $\Delta(H) \leq 10 d^2 (\log n)^{10}$.
	\end{lemma}

	\subsection{Expansion lemmas}\label{sec:technical-expansion}
Recall that we plan on combining two different types of expansions, one coming from the sublinear expansion and the other coming from the definition of $c(G)$. These two different reasons ensure expansion for sets with different sizes. To efficiently distinguish the expansions on different regimes of the sizes of sets, we introduce the following definition.
	
	\begin{defn}
	Suppose that $0<a\leq b\leq n$ and $\rho >0$.
 	We say that an $n$-vertex graph $G$ has $(a, b, \rho)$-expansion-property if every vertex set $X \subseteq V(G)$ with $|X| \in [a, b]$ satisfies $|N_G(X)| \geq \rho |X|$.
\end{defn}
	Note that an $n$-vertex $(\eps, k)$-robust-expander graph $G$ has $(k, m, \rho(m))$-expansion-property for any $k\leq m\leq n/2$. For this notion of expansion, we have a similar small diameter property as Lemma~\ref{short_path}.
\begin{lemma}\label{lem: short_path_general}
	Suppose that $G$ has $(a, \frac{n}{2}, \rho)$-expansion-property for some $\rho \in (0, 1]$.	Then for every sets $X, Y \subseteq V(G)$ with $|X|, |Y| \geq x \geq 2a$ and $|W| \leq \rho x/4$, there exists a path connecting $X$ and $Y$, which avoids $W$, having length at most $10\rho^{-1}\log(n/x)$. 
\end{lemma}
\begin{proof}
    For every $r \geq 0$, we have either $|B_{G-W}^r(X)| > n/2$ or 
    $$|B_{G-W}^{r+1}(X\setminus W)| \geq |B_{G-W}^{r}(X\setminus W)| + |N_G(B_{G-W}^{r}(X\setminus W))| - |W| \geq \left(1+\frac{\rho}{2}\right)|B_{G-W}^{r}(X\setminus W)|.$$
    Thus, there exists $r \leq 5\rho^{-1}\log(n/x)$ such that $|B_{G-W}^{r}(X\setminus W)| > n/2$.
\end{proof}

	In many cases, our plan for building units (or webs) are as follows. Let $W$ be the set of vertices in previously built units (or webs). 
	Take many disjoint stars (or units) and let $X$ be the vertices in them. While one star (or unit) $S$ is much smaller than $W$, the union $X$ is much larger than $W$ so that we can show $B^r_{G-W}(X)$ is large.
	As $B^{r}_{G-W}(X)$ is large, there must be many stars (or units) $S$ where $B^{r}_{G-W}(S)$ is large. Then these stars (or units) $S$ together with $B^r_{G-W}(S)$ are good building blocks for constructing additional units (or webs).
	
	In this subsection, we will prove several lemmas that provide a single small set $A_i$ from a large collection $A_1,\ldots,A_q$ such that $B^{r}_{G-W}(A_j)$ is large. In fact, for each $j$, we want to ensure $B^{r}_{G-W}(A_j)$ is large even when there is a specific set $W_j$ that $B^1_{G-W}(A_j)$ has to avoid.

\begin{lemma} \label{collectively_large}
Let $a,b,q,k \in \mathbb{N}$ and $x, \rho \in \mathbb{R}$ satisfying $a\leq kx \leq \frac{b}{3(1+\rho/4)}$. 
    Let $G$ be an $n$-vertex graph having $(a, b, \rho)$-expansion-property.
    Suppose that we have vertex sets $W,W_1,\dots, W_q$, and $A_1, \dots A_q \subseteq V(G)$ with $|W| \leq \frac{1}{2}\rho kx$ satisfying the following:
\stepcounter{propcounter}
\begin{enumerate}[label = {\bfseries \emph{\Alph{propcounter}\arabic{enumi}}}]
        \item For each $i\in [q]$, we have
        $|A_i|\geq x$ and $|W_i| \leq \frac{\rho x}{4}$.\label{mandu1}
        \item For each $i\in [q]$, we have $W_i \cap A_i = \varnothing$ and $W \cap A_i = \varnothing$.\label{mandu2}
        \item For each $I\subseteq [q]$ with $k\leq |I|\leq 3k$, we have $|\bigcup_{i \in I} A_i| \geq |I|x$.\label{collectively_large_B}
    \end{enumerate} 
    Then, there exists $J \subseteq [q]$ with $|J|\geq q-k$ such that the following holds.
\stepcounter{propcounter}
\begin{enumerate}[label = {\bfseries \emph{\Alph{propcounter}\arabic{enumi}}}]
        \item For each $j \in J$, we have $|B^1_{G - (W\cup W_j)}(A_j)| \geq (1+\rho/4)x$. \label{each_large}
        \item For each $I\subseteq J$ with $k\leq |I|\leq 3k$, we have $\left|\bigcup_{i \in I}B^1_{G - (W\cup W_i)}(A_i)\right| \geq (1+\rho/4)|I|x$. \label{collectively_large_2}
    \end{enumerate}
\end{lemma}
\begin{proof}
    For \emref{each_large}, it suffices to show that $|B^1_{G -(W \cup W_i)}(A_i)| < (1+\rho/4)x$ holds for at most $k$ indices $i\in [q]$. Suppose  there are at least $k$ such indices and let $K$ be a set of exactly $k$ such indices. For each $I\subseteq [q]$, we write $A_I =\bigcup_{i\in I} A_i$ and $W_I = \bigcup_{i\in I} W_i$.
		
    By the assumption and \emref{collectively_large_B}, we have 
    $$|A_K| \geq kx \geq a \enspace\text{ and }\enspace
    |A_K|  \leq \sum_{i \in K} |B^1_{G - (W \cup W_i) }(A_i)| <\Big(1+\frac{\rho}{4}\Big) kx \leq b.$$
    As $G$ has the $(a, b, \rho)$-expansion property, this ensures that 
    $$|N_{G - (W\cup W_K)}(A_K)| \geq |N_G(A_K)| - |W|- |W_K| \geq \rho kx - |W| -  |W_K|  \geq \rho kx/4.$$
    By averaging, there exists $i \in K$ such that 
    $|N_{G - (W\cup W_K)}(A_i)|\geq x \rho/4$. Hence we have
    $$|B^1_{G-(W\cup W_i)}(A_i)|\geq |B^1_{G-(W\cup W_K)}(A_i)| \geq |A_i| + x \rho /4  \geq (1+ \rho/4 ) x, $$ 
    a contradiction to the choice of $K$. 
    Therefore all but less than $k$ indices $i \in [q]$ satisfies 
    $|B^1_{G -(W \cup W_i)}(A_i)| \geq (1+\rho/4)x$.
    Let $J$ be the set of such $i\in [q]$ for which \emref{each_large} holds.

    To see \emref{collectively_large_2}, choose a subset $I \subseteq J$ of size at least $k$ and at most $3k$. 
    If $|A_I|  \geq b$, then \emref{collectively_large_2} is trivial as $b \geq (1+\rho/4) |I|x$. Otherwise, from the expansion property of $G$, \emref{mandu1} and \emref{collectively_large_B}, we have 
    $$\Big|\bigcup_{i\in I} B_{G- (W\cup W_i)}^1(A_i)\Big| \geq |A_I| + |N_{G}(A_I)|-|W|-|W_I| \geq |I|x + |I|\rho x  - |W|-|W_I| \geq |I|x (1+\rho/4).$$
    Hence \emref{collectively_large_2} holds.
\end{proof}
By repeatedly applying the above lemma, we can show that in an expander, among many disjoint sets $S_1, \cdots, S_q$, most of the sets have the exponential expansion, i.e. $B_{G-(W \cup W_i)}^r(S_i)$ grows exponentially in terms of $r$. This is captured in the following lemma.

\begin{lemma}\label{many_large_sets}
Let $q, k, r, x \in \mathbb{N}$ and $\rho \in \mathbb{R}$ with $q >2kr$, $a \leq kx < \frac{b}{20}$, and $k \geq (1+\rho/4)^r$.
 Suppose that $G$ has $(a, b, \rho)$-expansion property and we have vertex sets $W,W_1,\dots, W_q$, $A_1, \dots A_q \subseteq V(G)$ with $|W| \leq \frac{1}{2} \rho kx$ satisfying \emref{mandu1}--\emref{collectively_large_B}. 
    Then, there exists $J \subseteq [q]$ with $|J|\geq q-2kr$ such that for each $j \in J$, we have $|B^r_{G - (W\cup W_j)}(A_j)| \geq (1+\rho/4)^rx$.
\end{lemma}
\begin{proof}
 For each $0 \leq \ell \leq r$ and $i \in [q]$, let
 $$x_\ell:=(1+\rho/4)^\ell x, \enspace k_\ell:=\left\lceil \frac{k}{(1+\rho/4)^\ell} \right\rceil \enspace \text{ and } \enspace A_i^\ell := B^\ell_{G - (W\cup W_i)}(A_i).$$
 For $J\subseteq [q]$, let $A_J^{\ell} = \bigcup_{j\in J} A_j^{\ell}$.
 Let $I_0=[q]$ and we inductively construct $I_\ell$ satisfying the following.
 \stepcounter{propcounter}
  \begin{enumerate}[label = {\bfseries \Alph{propcounter}\arabic{enumi}}]
	\item$\mkern-10mu {}_{\ell}$ $I_\ell \subseteq I_{\ell-1}$ and $|I_\ell| \geq |I_{\ell-1}| - 2k$; \label{condition_2}
	\item$\mkern-10mu {}_{\ell}$ for each $i \in I_\ell$, we have $|A^{\ell}_i| \geq x_\ell$; \label{condition_3}
	\item$\mkern-10mu {}_{\ell}$ for each $J \subseteq I_{\ell}$ with $k_\ell \leq |J| \leq 3k_{\ell}$, we have $|A_J^{\ell}| \geq |J| x_\ell$. \label{condition_4}
  \end{enumerate}
  When $\ell=0$, \emref{mandu1} and \emref{collectively_large_B} imply that \ref{condition_3}$_0$ and \ref{condition_4}$_0$ hold.
  Now, assume that we already construct $I_{\ell}$ satisfying three conditions for some $\ell<r$.
  
  Note that $kx \leq x_\ell k_\ell\leq 2kx < \frac{b}{6}$. 
So we can apply Lemma~\ref{collectively_large} with $x_\ell$, $k_\ell$ and $A_i^\ell$'s for $i \in I_\ell$ playing the roles of $x,k$, and $A_i$, respectively, to obtain a subset $I'_{\ell+1} \subseteq I_\ell$ of size at least $|I_\ell|- k_{\ell} \geq |I_\ell|-k$ such that for each $i\in I'_{\ell+1}$, we have $|A^{\ell+1}_i|\geq x_{\ell+1}$, hence \ref{condition_3}$_{\ell+1}$ holds; and 
\begin{itemize}
    \item[$(\ast)$] for each $J\subseteq I'_{\ell+1}$ with $k_{\ell} \leq |J|\leq 3k_{\ell}$, we have $|\bigcup_{i\in J} A^{\ell+1}_i |\geq |J| x_{\ell+1}$. 
\end{itemize}

  If $I'_{\ell+1}$ satisfies \ref{condition_4}$_{\ell+1}$ for all sets $J \subseteq I'_{\ell+1}$ with $k_{\ell+1} \leq |J| \leq 3k_{\ell+1}$, then we are done by taking $I_{\ell+1}=I'_{\ell+1}$.
  If not, there is a set $J \subseteq I'_{\ell+1}$ with $k_{\ell+1} \leq |J| < 3k_{\ell+1}$ and $|A_J^{\ell+1}| < |J| x_{\ell+1}$.
  By ($\ast$), we can further assume that $|J| < k_\ell$.
  Let $I_{\ell+1} = I'_{\ell+1} \setminus J$.
  Then for any $J' \subseteq I_{\ell+1}$ with $k_{\ell+1} \leq |J'| < k_{\ell}$, we have $|J\cup J'|\geq 2k_{\ell+1} > k_{\ell}$, so ($\ast$) implies
  $$ |A_{J'}^{\ell+1}| \geq |A_{J \cup J'}^{\ell+1}| - |A_J^{\ell+1}| \geq |J \cup J'|x_{\ell+1} - |J|x_{\ell+1} \geq |J'|x_{\ell+1}.$$
  This together with ($\ast$) entails that \ref{condition_4}$_{\ell+1}$ holds with this choice of $I_{\ell+1}$. Note that $|I_{\ell+1}|\geq |I_\ell|-k_{\ell} - |J|\geq |I_{\ell}|-2k$, hence \ref{condition_2}$_{\ell+1}$ also holds. This proves the lemma.
\end{proof}
The following corollary provides a form which is convenient for us to use.
\begin{cor}\label{cor_single_large_set}
Let $G$, $a, b, x, \rho$, $A_i, W_i, W$ be as in Lemma~\ref{many_large_sets}, $q > 20k\log k/\rho$, $r \geq 10\log k/\rho$, and $0<\rho<1$.
 Then there exists $i \in [q]$ such that $|B_{G-(W \cup W_i)}^r(A_i)| \geq kx/2$.
\end{cor}
\begin{proof}
Let $r'$ be the largest $r$ such that $(1+\rho/4)^{r'}\leq k$ is true, then $(1+\rho/4)^{r'} \geq k/2$ as $(1+\rho/4)\leq 2$. 
 Then as $q-2kr'\ge 1$, by Lemma \ref{many_large_sets}, there exists an index $i \in [q]$ such that $|B_{G-(W \cup W_i)}^{r'}(A_i)| \geq x(1+\rho/4)^{r'} \geq kx/2$.
As $(1+\rho/4)^r \geq k$, we have $r'\leq r$, hence $|B_{G-(W \cup W_i)}^r(A_i)| \geq |B_{G-(W \cup W_i)}^{r'}(A_i)| \geq kx/2$.
\end{proof}

The following lemma will be useful for building webs later. 
This shows that even if our expansion is weak, if the sets $W_i$ we want to avoid are not so congested around $A_i$, then we can find an index $i$ such that $B_{G-(W \cup W_i)}^r(A_i)$ is large.

\begin{lemma}\label{single_large_set}
Let $\rho\in (0,1)$ and suppose $G$ is an $n$-vertex graph with $(a, b, \rho)$-expansion-property and $q, x\in \mathbb{N}$ with $a \leq qx \leq \frac{b}{4}$.
 Let $W, W_1, W_2, \cdots, W_q,$ $A_1, A_2, \cdots, A_q$ be sets of vertices with $|A_i| \geq x$. Let $A_i^\ell = B_{G-(W \cup W_i)}^{\ell}(A_i)$ and suppose
 \stepcounter{propcounter}
 \begin{enumerate}[label = {\bfseries \emph{\Alph{propcounter}\arabic{enumi}}}]
         \item $A_i \cap A_j = \varnothing$ if $i \neq j$.\label{A-disj}
		 \item $A_i \cap W_i = \varnothing$ and $A_i \cap W = \varnothing$ for all $i \in [q]$.
		 \item $|W| \leq \frac{1}{2} \rho qx$ and for every $\ell \geq 1$, we have $|W_i \cap N_{G-W}(A_i^{\ell-1})| \leq \frac{1}{40}\ell x\rho^2$.\label{eq: thin layer}
 \end{enumerate} 
 Then, for $r = \frac{50\log q}{\rho}$, there exists $i \in [q]$ such that $|A_i^{r}| \geq qx$.
\end{lemma}
\begin{proof}
Assume that we have $|A_i^r|<qx$ for each $i\in [q]$. We will derive a contradiction from this assumption to prove the lemma.

For each $I\subseteq [q]$ and $\ell \in [r]$, let $A_I^{\ell} = \bigcup_{i\in I} A_i^{\ell} \enspace \text{and} \enspace W_I = \bigcup_{i\in I} W_i.$
 By deleting vertices if necessary, we assume that $|A_i|=x$ for all $i\in[q]$ and so by \emref{A-disj}, $|A_{[q]}|=qx$. Let $I_0=[q]$ and $r_0=0$. Consider a collection $I_0,\dots, I_s$ of nonempty sets and numbers $r_0,\dots, r_s$ satisfying the following for all $\ell\in [s]$. Among all possible collections, we consider one with maximum $s$. 
 \stepcounter{propcounter}
 \begin{enumerate}[label = {\bfseries \Alph{propcounter}\arabic{enumi}}]
	\item $|I_{\ell}| \leq \frac{3}{4}|I_{\ell-1}|$ if $\ell \geq 1$. \label{condition_size}
    \item $r_{\ell} \leq r_{\ell-1} + \frac{10}{\rho}$ if $\ell\geq 1$. \label{condition_rell}
	\item $|A_{I_\ell}^{r_\ell}| \geq \max\{ qx, ( \frac{1}{4} \rho r_\ell + 1) x|I_{\ell}| \}$. \label{condition_5}
	\item $|A_{I_\ell}^{r_\ell}| \leq 4 qx $. \label{condition_6}
 \end{enumerate}
 Such a choice of $I_0,\dots, I_s$ exists as the trivial collection $\{ I_0 \}$ itself satisfies \ref{condition_5} and \ref{condition_6} for $\ell=0$. 

For such a maximum choice, \ref{condition_size} and \ref{condition_rell} ensures that $r_s\leq \frac{50\log{q}}{\rho}= r$. Hence, our initial assumption together with \ref{condition_5} implies that $|I_s|>1$. We will derive a contradiction to the maximality of $s$. Let $\ell \in \mathbb{N}$ be the minimum number satisfying $|A^{r_{s}+\ell}_{I_s}| > 4qx,$
if it exists; otherwise let it be $\infty$.

We first show that for each $0 \leq \ell' \le \min\{\ell, \frac{10}{\rho} \}$, we have
\begin{align}\label{eq: expand rho4 ell'}
    |A_{I_s}^{r_s+\ell'}|\geq (1+\rho/4)^{\ell'}|A_{I_s}^{r_s}|.
\end{align}
If $\ell'=0$, this is vacuous. Otherwise, $qx \leq |A_{I_s}^{r_s}| \leq |A_{I_s}^{r_s+\ell'-1}|\leq  4qx$ and $[qx,4qx]\subseteq [a,b]$, $(a,b,\rho)$-expansion property of $G$ implies $|N_{G-W}(A^{r_s+\ell'-1}_{I_s})| \geq \rho|A^{r_s+\ell'-1}_{I_s}|-|W| \geq \frac{\rho}{2}|A^{r_s+\ell'-1}_{I_s}|$. With this and \emref{eq: thin layer}, we obtain
\begin{align*}
    |A^{r_s+\ell'}_{I_s}| &\geq  |A_{I_s}^{r_s+\ell'-1}| + |N_{G-W}(A^{r_s+\ell'-1}_{I_s})| - \sum_{i\in I_{s}} \left|W_i \cap N_{G-W}( B^{r_{s}+\ell'-1}_{G-(W\cup W_i)}(A_i))\right|  \\
    &\geq \left(1+\frac{\rho}{2}\right)|A_{I_s}^{r_s+\ell'-1}| - \frac{1}{40} (r_{s}+\ell') x \rho^2 |I_s| \\
    &\geq \left(1+\frac{\rho}{4}\right)|A_{I_s}^{r_s+\ell'-1}| + \frac{\rho}{4}|A_{I_s}^{r_s}|  - \frac{1}{40} (\rho r_{s} + \ell'\rho) x \rho |I_s| \\
    &\geq \left(1+ \frac{\rho}{4}\right)|A_{I_s}^{r_s+\ell'-1}|.
\end{align*}
The final inequality holds by \ref{condition_5} and the assumption $\rho\ell' \leq 10$.
This verifies \eqref{eq: expand rho4 ell'}. If $\ell> \frac{10}{\rho}$, this would imply $|A^{r_s+ \frac{10}{\rho}}_{I_s}| \geq (1+ \frac{\rho}{4})^{\frac{10}{\rho}} qx > 4qx$, contradicting the definition of $\ell$. Thus, $\ell\leq \frac{10}{\rho}$.

Let $r_{s+1} = r_s + \ell$, then \ref{condition_rell} holds for $s+1$, and by definition of $\ell$, $|A^{r_{s+1}}_{I_s}|\ge 4qx$.
From our initial assumption, we know $|A^{r_{s+1}}_j|\leq qx$ for each $j\in I_s$.
As $\sum_{j\in J}|A^{r_{s+1}}_j| \geq |A^{r_{s+1}}_J|$ holds for any index set $J\subseteq [q]$,  we can partition $I_s$ into $I'_1,\dots, I'_B$ with $qx\leq |A^{r_{s+1}}_{I'_k}| \leq 3qx$ for all $k\in [B]$.
To see this, we can add an index one by one to $I'_1$ until $|A^{r_{s+1}}_{I'_1}|\geq qx$. Note that $qx\leq |A^{r_{s+1}}_{I'_1}|\leq 2qx$. Then add remaining indices one by one to $I'_2$ until $qx\leq |A^{r_{s+1}}_{I'_2}|\leq 2qx$. Repeating this until the union of sets corresponding to the remaining indices has size less than $qx$, add them to an arbitrary part, then we can ensure $qx\leq |A^{r_{s+1}}_{I'_k}|\leq 3qx$ for all $k\in [B]$. 

As $\sum_{k\in [B]}|A^{r_{s+1}}_{I'_k}|\geq |A^{r_{s+1}}_{I_s}|,$ 
there exists $I_{s+1}\in \{I'_1,\dots, I'_B\}$ such that 
$$\frac{3}{4} |A^{r_{s+1}}_{I_s}| \geq 3qx\geq |A_{I_{s+1}}^{r_{s+1}}| \geq \frac{|I_{s+1}|}{|I_{s}|} \cdot |A^{r_{s+1}}_{I_s}|.$$
This choice of $I_{s+1}$ satisfies \ref{condition_size} and \ref{condition_6}. As $I_{s}$ satisfies \ref{condition_5}, by \eqref{eq: expand rho4 ell'}, we have 
\begin{align*}
|A_{I_{s+1}}^{r_{s+1}}| &\geq \frac{|I_{s+1}|}{|I_{s}|} \cdot \left(1+\frac{\rho}{4}\right)^{\ell} |A_{I_s}^{r_s}| \geq \left(1+\frac{\rho}{4} \right)^{\ell} \cdot  \left(\frac{1}{4}\rho r_s + 1\right) x|I_{s+1}| \\
&\geq \left(1+ \frac{\ell \rho}{4}\right) \left(\frac{1}{4}\rho r_s + 1 \right) \cdot  x|I_{s+1}| 
\geq \left( \frac{1}{4} \rho r_s+ \frac{1}{4}\rho \ell + 1\right) \cdot  x|I_{s+1}| =  \left(\frac{1}{4} \rho r_{s+1} + 1\right) x|I_{s+1}|.
\end{align*}
Thus $I_{s+1}$ satisfies \ref{condition_5}.
Therefore, we obtained a set $I_{s+1}$ satisfying \ref{condition_size}--\ref{condition_6}, a contradiction to the maximality of $s$. This proves that there must be an index $i\in [q]$ with 
$|A_i^r|\geq qx$.
\end{proof}

\subsection{Building webs and subdivisions}
In this subsection, we state and prove lemmas useful for constructing a desired subdivision. 
In the following three lemmas, there are complicated-looking conditions that several parameters have to satisfy.
This is necessary as we need to apply these lemmas to graphs with different expansion properties. For ease of verification, in actual applications, the parameters will be chosen so that the inequalities hold with large gaps.

\begin{lemma}\label{lem_units}
Let $\rho\in (0,1)$ be a real number and suppose $0<1/T \ll 1$. Let $n,p,x,k,h_1,h_2,h_3$ be integers larger than $T$. Suppose the following hold.
\begin{align}
   x &\geq  \frac{200 ph_1h_3}{\rho^2 k} + h_2, &     p&\geq 4kh_3 + 2\rho k h_1,\label{eq: conditions on lem_units1}\\
   &4\left( \frac{nh_1}{px} \right)^3 \leq k \leq \frac{n}{30x}, &    h_3 &\geq \frac{30}{\rho} \log\left(\frac{10 n}{\rho k x} \right).\label{eq: conditions on lem_units2}
\end{align}
 Let $G$ be an $n$-vertex graph with $(\frac{kx}{2}, \frac{n}{2}, \rho)$-expansion-property and $W \subseteq V(G)$ be a set of vertices of size at most $\frac{1}{100} \rho^2 kx$.
 If $G-W$ contains $2p$ vertex-disjoint stars each with $x$ leaves, then $G-W$ contains an $(h_1, h_2, h_3)$-unit.
\end{lemma}
\begin{proof}
	Let $S_1, S_2, \cdots, S_p, R_1, R_2, \cdots, R_p$ be $2p$ vertex-disjoint stars each having exactly $x$ leaves. Let $C$ be the set of center vertices of these stars. Note that the existence of such stars implies $px \leq n$.
	Let $I\subseteq [p]\times [p]$ 
	and $\mathcal{P} = \{P_{i,j}: (i,j)\in I\}$ be a collection of paths where $I$ is a maximal set among all $(I,\mathcal{P})$ satisfying the following.
 \begin{enumerate}[label = {\bfseries I\arabic{enumi}}]
		\item\label{cond: I1} For each $(i,j)\in I$, the path $P_{i,j}$ avoids the vertices in $C$ and has length at most $h_3-2$ connecting a leaf of $S_i$ and a leaf of $R_j$. 
		\item\label{cond: I2} For two distinct tuples $(i,j),(i',j')\in I$, $P_{i,j}$ and $P_{i',j'}$ are vertex-disjoint.
	\end{enumerate}
	For $\ell\in [2]$ and $i\in [p]$, let $w_\ell(i)$ be the number of tuples in $I$ whose $\ell$-th coordinate is $i$.
	If some $\ell\in [2]$ and $i\in [p]$ satisfies $w_\ell(i)\geq h_1$, then take those paths $\{ P_{i,j}: (i,j)\in I\}$ if $\ell=1$ or $\{P_{j,i}: (j,i)\in I\}$ if $\ell=2$. 
	These paths together with the corresponding stars form a desired unit. Note that in this unit, some paths may use up leaves from some stars. However, each star still has $x - h_1h_3$ leaves disjoint from the paths. As $p\geq \rho kh_1\geq \rho k$ and $\rho<1$, we have $x \geq \frac{200ph_1h_3}{\rho^2 k} + h_2 \geq h_1h_3 + h_2.$
	Hence, $x-h_1h_3\geq h_2$ and the unit we obtained is an $(h_1,h_2,h_3)$-unit as desired.

	Now we assume that for all  $\ell\in [2]$ and $i\in [p]$, we have
	\begin{align}\label{eq: at most h1}
	    w_{\ell}(i) < h_1.
	\end{align}
	We will show that this assumption yields a contradiction to the maximality of $I$ by finding a tuple $(i,j)$ and a path $P_{i,j}$ which we can add to $I$ and $\mathcal{P}$ while keeping the condition \ref{cond: I1} and \ref{cond: I2}.
	
For each $i\in [p]$, let
$P_i = \bigcup_{j:(i,j)\in I} V(P_{i,j})$, $P=\bigcup_{i\in [p]} P_i$ and $\widetilde{W} = W\cup P \cup C.$
As each $V(P_{i, j})$ contains at most $h_3$ vertices, \eqref{eq: conditions on lem_units1} together with~\eqref{eq: at most h1} imply that 
\begin{align}\label{eq: PiPQ size}
\begin{split}
    |P_i|&\leq h_1 h_3, \enspace \enspace   |P|\leq ph_1h_3, \enspace \text{and} \enspace \\ |\widetilde{W}|&\leq \frac{1}{100}\rho^2 kx + ph_1h_3 + 2p \leq \frac{1}{100}\rho^2 kx +\frac{1}{200}\rho^2 kx + 2p \leq \frac{1}{50}\rho^2 kx.
    \end{split}
\end{align}
Let $r_1 = \frac{10}{\rho} \log\left(\frac{2nh_1}{px}\right) \enspace \text{ and } \enspace r_2 = h_3 -r_1 -2.$
Note that \eqref{eq: conditions on lem_units1} implies that we have 
\begin{align}\label{eq: nrhokx}
    \frac{n}{\rho kx} \geq \frac{n}{px/(2h_1)}= \frac{2nh_1}{px}.
\end{align}
As $px\leq n$, we get $\frac{nh_1}{px} \geq 100$. This together with \eqref{eq: conditions on lem_units2} implies that 
\begin{align}\label{eq: at most k/2}
    \frac{4 nh_1}{px}  \leq   (1+\frac{\rho}{4})^{r_1}  \leq 2\left( \frac{nh_1}{px} \right)^3 \leq k/2.
\end{align}
On the other hand, we also have
\begin{align}\label{eq: at least n/rhokx}
r_2 \geq \frac{30}{\rho} \log \left(\frac{10n}{\rho kx}\right) - \frac{10}{\rho} \log\left(\frac{2nh_1}{px}\right) -2 
\geq \frac{10}{\rho} \log \left(\frac{10 n}{\rho kx}\right).
\end{align}
The final inequality holds by \eqref{eq: nrhokx}.

For each $i\in [p]$, consider $B^{r_1}_{G-\widetilde{W}}(V(S_i)\setminus \widetilde{W})$.
If we can find a path of length $r_2$ from this set to some $R_j$ avoiding $\widetilde{W}$, then we obtain a path between $S_i$ and $R_j$. However, this does not yet yield a contradiction to the maximality of $I$ as the tuple $(i,j)$ might already be in $I$. In order to find $(i,j)$ not in $I$, we define $Q$ as follows:
$$Q = \{ v\in V(G)\setminus\widetilde{W}: v\in B^{r_1}_{G-\widetilde{W}}(V(S_i)\setminus \widetilde{W}) \text{ for at least $h_1$ choices of } i\in [p] \}.$$
We claim that
\begin{align}\label{eq: Q size}
       |Q| \geq \frac{1}{10} \rho k x.
\end{align}
Indeed, suppose that \eqref{eq: Q size} does not hold. 
We aim to apply Lemma~\ref{many_large_sets} to obtain vertices outside $Q$ belonging to $B^{r_1}_{G-\widetilde{W}}(V(S_i)\setminus \widetilde{W})$ for at least $h_1$ choices of $i\in [p]$ to derive a contradiction.

Let $W' = \widetilde{W}\cup Q$, then by \eqref{eq: PiPQ size}, we have $|W'|\leq \frac{1}{5}\rho kx$. 
Let $x'=x/2$ and we count the sets $V(S_i)\setminus W'$ having at least $x'$ vertices. As $S_1,\dots, S_p$ are vertex-disjoint, $|V(S_i)\cap W'|\geq x/2$ for at most $\frac{2|W'|}{x}\leq k$ many choices of $i\in [p]$. 
Hence, by letting $I' \subseteq [p]$ be the set of $i\in [p]$ with $|V(S_i)\setminus W'| > x/2$, we have $|I'|\geq p - k \geq 3kh_3$.

As $G$ has $(\frac{kx}{2},\frac{n}{2},\rho)$-expansion property and $\frac{kx}{2}\leq kx' \leq \frac{n}{40}$, the second and third inequalities of \eqref{eq: at most k/2} allow us to apply Lemma \ref{many_large_sets} with $|I'|, \{S_i\setminus W' : i\in I'\}, W', \varnothing$ and $r_1$ playing the roles of $q, \{A_i:i\in [q]\}, W, W_i$ and $r$, respectively. So there exists $I'' \subseteq [p]$ with $|I''| \geq p-2k-2kr_1$ such that for all $i \in I''$, we have $|B_{G-W'}^{r_1}(V(S_i) \setminus W')| \geq x' \Big(1+ \frac{\rho}{4}\Big)^{r_1}.$
   However, by the first inequality of \eqref{eq: at most k/2}, this means 
	\begin{align*}
	   \sum_{i \in [p]}|B_{G-W'}^{r_1}(V(S_i) \setminus W')| & \geq |I''| \cdot x' \Big(1+ \frac{\rho}{4}\Big)^{r_1} \geq nh_1.
	\end{align*}
The final inequality holds as \eqref{eq: conditions on lem_units1},  \eqref{eq: conditions on lem_units2}  and \eqref{eq: nrhokx} imply 
$|I''|\geq p-2k-2kr_1 \geq p - 2k -kh_3 \geq p/2$.
   Therefore, there exists a vertex $v$ out of $W'$, hence also $Q$, belonging to $B_{G-W'}^{r_1}(V(S_i) \setminus W')$ for at least $h_1$ choices of $i\in [p]$, contradicting the definition of $Q$. This proves \eqref{eq: Q size}.

By \eqref{eq: PiPQ size} and \eqref{eq: Q size}, we know that 
$$ |\widetilde{W}| \leq \frac{1}{50}\rho^2 kx <
\frac{\rho}{4} \cdot \frac{1}{10}\rho kx  \leq \frac{\rho}{4} \min\{ |Q|,\big|\bigcup_{j\in [p]} R_j\setminus C\big|\}.$$
Here, the final inequality holds as $\rho kx< px \leq |\bigcup_{j\in [p]} R_j\setminus C|$ by \eqref{eq: conditions on lem_units1}.
Hence, by Lemma~\ref{lem: short_path_general} and \eqref{eq: at least n/rhokx}, there exists a path $\widetilde{P}$ connecting $Q$ and $\bigcup_{j\in [p]} R_j\setminus C$ while avoiding all vertices in $\widetilde{W}$ which has length at most $10\rho^{-1}\log(\frac{n}{\rho kx/10}) \leq r_2$.

	Let $v$ be an endpoint of $\widetilde{P}$ in $Q$ and $R_j$ be the star that contains another endpoint of $P$.
	By the definition of $Q$, there are $h_1$ indices $i \in [p]$ such that there exists a path connecting $v$ and $S_i$ of length at most $r_2$ which avoids $\widetilde{W}$.
	As we assume that $w_2(j) < h_1$, there are less than $h_1$ choices of $i$ with $(i,j)\in I$. Hence we can choose $i$ such that $(i,j)\notin I$ and there exists a path connecting $v$ and $S_i$ of length at most $r_2$ avoiding $\widetilde{W}$.
	By combining $\widetilde{P}$ and this path, we obtain a walk connecting $S_i$ and $R_j$ while avoiding $\widetilde{W}$.
	This contains a path $P_{i,j}$ connecting $S_i$ and $R_j$ of length at most $h_3-2$ avoiding $\widetilde{W}$. Then $(I\cup \{(i,j)\},\mathcal{P}\cup \{P_{i,j}\})$ satisfies \ref{cond: I1} and \ref{cond: I2}, contradicting the maximality of $I$.
   \end{proof}
   
   We can obtain many units using the previous lemma. The following lemma provides a way to assemble these units into a web. Note that we have two different parameters $\rho_1$ and $\rho_2$ here. When we use this lemma in the proof of Lemma~\ref{lem_nc_expander}, the  $(px,4px,\rho_1)$-expansion will be obtained from the definition of crux and $(px,\frac{n}{2},\rho_2)$-expansion will be obtained from the definition of the sublinear expander.
   \begin{lemma}\label{lem_webs}
   Suppose $0<1/T, c\ll 1$ and let $p, q, x, h_1, h_2, h_3, h_4, h_5$ be integers larger than $T$ and $\rho_1, \rho_2 \in (0, 1]$. 
Suppose that the following hold.
   \begin{align*}
h_1 \geq 4 h_5,  \enspace h_4 \leq \min\Big\{  cx \rho_1^2, \frac{1}{8}q\Big\}, \enspace
h_5 \geq h_3 + 50\rho_1^{-1}\log q + 10 \rho_2^{-1}\log\left(\frac{2n}{px}\right) + 10\rho_2^{-1}\log\left(\frac{4n}{qh_1h_2}\right).
\end{align*}
   	Let $G$ be an $n$-vertex graph and $W \subseteq V(G)$ is a set of vertices satisfying 
\begin{align} |W| \leq \frac{1}{32}\rho_1 p x - p \enspace \text{and} \enspace |W|+p+2h_4h_5 \leq \rho_2 \cdot\min\Big\{\frac{1}{16}qh_1h_2, \frac{1}{8}px\Big\}.\label{eq: lem_webs condition 2}
\end{align}
Suppose $G$ has $(px, 4px, \rho_1)$-expansion-property and $(px, \frac{n}{2}, \rho_2)$-expansion-property.
	If $G-W$ contains pairwise vertex-disjoint subgraphs $S_1, \dots, S_p, R_1, \dots, R_q$ where each $S_i$ is a star with $x$ leaves and each $R_j$ is an $(h_1, h_2, h_3)$-unit for all $i \in [p]$ and $j \in [q]$. Then $G-W$ contains a $(h_4, h_5, h_1/2, h_2/2, h_3)$-web which contains a center vertex of $S_1, \dots, S_p$ only at its center and whose exterior is contained in $\bigcup_{j \in [q]} V(R_j)$.
   \end{lemma}
   \begin{proof}
	For each $i\in [p]$ and $j\in [q]$, let $c_i$ be the center vertex of $S_i$, $c'_j$ be the center vertex of $R_j$, and let $S'_i = V(S_i) \setminus\{c_i\}$. Let $C$ and $C'$ be the set of center vertices of stars and units, respectively, and let $W'=W \cup C$.
	For each $i \in [p]$, let $\mathcal{P}_i=\{P_{i, 1}, \cdots, P_{i, j_i}\}$ be a maximal collection of paths satisfying the following for each $i\in [p]$ and $j\in [j_i]$.
	\begin{enumerate}[label = {\bfseries P\arabic{enumi}}]
		\item $P_{i, j}$ is a path of length at most $h_5$ connecting $c_i$ and $c'_k$ for some $k=k(i, j) \in [q]$ where $k(i,j)\neq k(i,j')$ for all $j'<j$. \label{cond_P1}
		\item $P_{i, j}$ is a shortest path between $c_i$ and $C' \setminus \{c'_{k(i, j')} \mid j'<j\}$ in the graph $G -(W'\setminus\{c_i\})- \bigcup_{j'<j} V(P_{i, j'})$.\label{cond_P2}
	\end{enumerate}
	For each path $P_{i,j}$, we call $c_i$ the first vertex and $c'_{k(i,j)}$ the last vertex.
	If $j_i \geq 2h_4$ for some $i\in [p]$, then the paths $P_{i,1},\dots, P_{i,2h_4}$ together with the units $R_{k(i,1)},\dots, R_{k(i,2h_4)}$ yield a desired web.
	Indeed, count the units among $R_{k(i,1)},\dots, R_{k(i,2h_4)}$ which contain at least $h_1/2$ vertices of $U:=\bigcup_{j \in [2h_4]} V(P_{i,j})$. As $|U| \leq 2 h_4 h_5,$
	we get that at most $\frac{4h_4h_5}{h_1} \leq h_4$ units contain at least $h_1/2$ vertices of $U$. Hence, there are at least $h_4$ units each containing at most $h_1/2$ vertices of $U$. We delete all branches intersecting $U$. As each unit has $h_1$ branches, each of the remaining units contains an $(h_1/2,h_2, h_3)$-unit.
	These units together with the corresponding paths yield a desired $(h_4, h_5, h_1/2, h_2, h_3)$-web.

	Assume that $j_i <2h_4$ for all $i$ and let $P_i$ be the set of vertices in paths in $\mathcal{P}_i$ except $c_i$. Note that we have $|P_i|\leq 2h_4h_5$.
    We aim to apply Lemma~\ref{single_large_set}. Let $x'= x/2$. 
    \ref{cond_P2} implies that $P_{i,j}$ intersect with $V(S_i)$ at only the center $c_i$ and at most one leaf; and as $|W|\leq \frac{1}{32}\rho_1 px - p$, we get
    $$|S'_i \setminus P_i|\ge x-2h_4 \geq x' \enspace \text{ and } \enspace   |W' | = |W|+|C| \leq \frac{1}{16}\rho_1 p x' .$$
    Moreover \ref{cond_P2} ensures that for each $\ell \geq 1$, every vertex in
    $ N_{G-W'}\left(B_{G-(W' \cup P_i)}^{\ell-1}(S'_i \setminus P_i)\right) \cap P_i$ is one of the first $\ell+2$ vertices of paths in $\mathcal{P}_i$. Indeed, if that's not the case, then there exists a path $Q$ of length $\ell$ from $S'_i \setminus P_i$ to a vertex $v$ in some $P_{i,j}$, where $v$ is not one of the first $\ell+2$ vertices. Then $Q$ together with the vertices in $P_{i,j}$ after $v$ contains a path shorter than $P_{i,j}$ avoiding all vertices in $W$ and the vertices in $\bigcup_{j'<j} V(P_{i,j'}).$ This is a contradiction to \ref{cond_P2}.
    Hence, as $h_4 \leq  cx\rho_1^2$, we have
        \begin{align*}
 \big|N_{G-W'}\left(B_{G-(W' \cup P_i)}^{\ell-1}(S'_i \setminus P_i)\right)\cap P_i\big| \leq (\ell+2) 2h_4 \leq \frac{1}{40} \ell x' \rho_1^2.
    \end{align*}
    As $G$ has $(px,4px,\rho_1)$-expansion property, 
we can apply Lemma~\ref{single_large_set} with the sets $W', P_1,\dots, P_p,$ $S_1\setminus P_1,\dots, S_p\setminus P_p$ playing the roles of $W, W_1,\dots, W_q,$ $A_1,\dots, A_q$, respectively. 
Then we obtain an $i\in [p]$ such that  for $r = \frac{50 \log q}{\rho_1}$, $|B_{G-W'-P_i}^{r}(V(S_i) \setminus P_i)| \geq px.$
      
   For this choice of $i$, let $J=\{ k(i,j'): j'\leq j_{i}\}$ and $J' = \{k \in [q]: |V(R_{k}) \cap P_i| >h_1/2 \}$. As $|J|\leq 2h_4$ and $|J'| \leq \frac{|P_i|}{h_1/2}\leq \frac{4h_4h_5}{h_1} \leq h_4$, we have 
   $|[q]\setminus (J \cup J')| \geq q/2$.
   Let $R'_j$ be the subgraph of $R_j$ obtained by deleting branches of $R_j$ that intersect with $P_i$ and let $R =\bigcup_{j\in [q]\setminus (J\cup J')} V(R'_j)$.
   For each $j\in [q]\setminus(J\cup J')$, $R_j$ loses at most $h_1/2$ branches, so it still has at least $\frac{1}{2} h_1h_2$ vertices. Thus, $|R| \geq \frac{1}{4} q h_1h_2.$
   
As $|W'| + |P_i| \leq |W|+p+2h_4h_5 \stackrel{\eqref{eq: lem_webs condition 2}}{\leq} \frac{1}{4}\rho_2 \min\{ |R|, |B_{G-W'-P_i}^r(V(S_i) \setminus P_i) \},$ 
Lemma~\ref{lem: short_path_general} implies that there exists a path $P^*$ between $R$ and $B_{G-W'-P_i}^{r}(V(S_i) \setminus P_i)$ with 
\begin{align*}
|P^*| & \leq 10\rho_2^{-1}\min \left\{\log\left(\frac{2n}{px}\right), \log\left(\frac{4n}{qh_1h_2}\right)\right\} \\
& \leq \rho_2^{-1}\left( 10 \log\left(\frac{2n}{px}\right) + 10 \log\left(\frac{4n}{qh_1h_2}\right) \right) -2\leq h_5 - r-h_3-2.
\end{align*}
The second inequality comes from the fact that the existence of $R_1, \dots, R_q$ implies $n \geq qh_1h_2$.
Let $k \in [q] \setminus (J \cup J')$ be the index that the end point of $P^*$ lies in.
Concatenating $P^*$ with a path from $S_i$ within $B_{G-W'-P_i}^{r}(V(S_i) \setminus P_i)$ of length at most $r$ and a path from $c'_k$ within $R_k$ of length at most $h_3+1$, we obtain a path from $S_i$ to $c'_k$ for some $k\in [q]\setminus J$ of length at most $h_5-1$ disjoint from $W' \cup P_i$. 
This shows that there exists a path from $c_i$ to $C' \setminus \{c'_{k(i, j')} \mid j'<j\}$ avoiding $(W' \setminus\{c_i\})\cup P_i$, of length at most $h_5$ as $c_k' \in C' \setminus \{c'_{k(i, j')} \mid j'<j\}$. Then choose a shortest path among all such paths, which satisfies \ref{cond_P1} and \ref{cond_P2}.
Adding this path to $\mathcal{P}_i$ contradicts the maximality of $\mathcal{P}_i$. 
This proves that there must be an $i\in [p]$ with $j_i\geq 2h_4$, providing an $(h_4, h_5, h_1/2, h_2/2, h_3)$-web.
Moreover, as $C \subseteq W'$, the obtained web does not contain any vertex of $C$ except at its center. Also, the exterior of this web consists of exterior vertices of $R_1, \dots, R_q$ as desired.
   \end{proof}
   
   Finally, the following lemma provides a way to combine webs to obtain a clique subdivision.
	\begin{lemma}\label{lem_conn_webs}
	Suppose $0<1/T \ll 1$.
	 Let $n, t, h_1, h_2, h_3, h_4, h_5$ be integers larger than $T$ and $\rho \in (0, 1]$ be a real and $s\leq \min\{\frac{t}{2}, \frac{h_4}{30}\}$. Suppose the following holds.
	 \begin{align*}
	     h_5 \geq h_3 + \rho^{-1} \log\left(\frac{8n}{h_1h_2h_4}\right), \enspace 
	     \rho h_1h_2h_4 \geq 50( h_4h_5 t + 20h_5 st), \enspace \text{ and }\enspace 
	     h_1h_4 \geq 400 h_5 s.
	 \end{align*}
	 If an $n$-vertex graph $G$ has the $(\frac{h_1h_2h_4}{8}, \frac{n}{2}, \rho)$-expansion-property and has $t$ internally disjoint $(h_4, h_5, h_1, h_2, h_3)$-webs.
	 Then $G$ contains a $K_s$-subdivision.
	\end{lemma}
	\begin{proof}
		By decreasing $t$, we may assume that $t=2s$ and $t\leq h_4/15$.
		Let $U_1, U_2, \cdots, U_t$ be an ordering of internally disjoint $(h_4,h_5,h_1,h_2,h_3)$-webs.
		We will inductively connect two webs by a path satisfying the following.
		\begin{enumerate}
			\item Among all not-deleted webs, choose two webs $U_i$ and $U_j$ from first $s$ webs that are not connected by a path. Find a path of length at most $10\rho^{-1} \log(\frac{8n}{h_1h_2h_4})$ between them which is disjoint from all cores of webs and previously built paths.
			\item Extend the path obtained in the previous step using edges of $U_i$ and $U_j$ to get a path of length at most $10\rho^{-1} \log(\frac{8n}{h_1h_2h_4}) + 2h_3+2h_5$ connecting the centers of two webs.
			\item For each branch of units within a not-deleted web, if more than half of the leaves of a star belong to previously-built paths or a vertex in the path connecting the center of the unit and the center of the star belongs to a previously-built path, we delete the branch.
			\item For each unit within a not-deleted web, if more than half of the branches are deleted, we delete the unit.
			\item For each web, if more than half of the units within the web are deleted, then we delete the web.
	   \end{enumerate}
	   Note that the above process connected two webs within the first $s$ webs. Hence, each web is connected to at most $s$ webs before it. So the above process ends within $st$ rounds. 
	   
	   If a web is not deleted, then it has at least $h_4/2$ units each of which has at least $h_1/2$ branches where each branch has at least $h_2/2$ leaves at the end. Hence, its exterior has size at least $\frac{1}{8}h_1h_2h_4$.
       The total size of cores of webs is at most 
		   $t(1+h_4h_5)$.
	  As less than $st$ paths were built during the whole process, the total number of vertices belonging to one of the previously-built paths is at most $st\left(10\rho^{-1} \log\left(\frac{8n}{h_1h_2h_4}\right) + 2h_3 + 2h_5\right) \leq 20st h_5$.
	   Hence, the number of vertices within a core or previously-built path is at most $20st h_5 + t(1+h_4h_5) \leq \frac{1}{4}\rho \cdot \frac{1}{8}h_1h_2h_4$.
	   Therefore, we can connect the exteriors of two webs by a path of appropriate lengths using Lemma~\ref{lem: short_path_general} and the $(\frac{h_1h_2h_4}{8}, \frac{n}{2}, \rho)$-expansion-property while avoiding core of webs and all previously used path.
	   This shows that the process only ends when either all first not-deleted $s$ webs are connected pairwise, providing a $K_s$-subdivision, or ends after deleting more than $t-s$ webs.
	  
	  Let $P$ be the set of vertices in the paths we obtain from the process.
	  As the exterior of a web contains at most $20st h_5$ vertices of $P$, this can delete at most $\frac{40st h_5}{h_1}$ stars within a web, causing at most 
	  $\frac{80st h_5}{h_1h_4}$ units to be deleted.
	  
	  A web is deleted if $h_4/2$ units are deleted. As the paths do not intersect the cores of webs,  a web is deleted only when at least $\frac{h_4}{2} - \frac{80st h_5}{h_1h_4}-s$ units within the web contain at least $h_1/2$ vertices of $P$ in its interior.
	  This implies that each deleted web contains at least $\frac{h_1h_4}{4} - \frac{40st h_5}{h_4}- \frac{sh_1}{2} \geq \frac{h_1h_4}{5}$
	  vertices of $P$ within its interior. Here, we used the assumption that $2s=t\leq h_4/15$ to obtain the final inequality.
	  
	   As all webs are internally disjoint, this implies that at most $\frac{|P|}{h_1h_4/5} \leq \frac{100 st h_5}{h_1h_4} \leq \frac{t}{2}$ webs are deleted.
	   The final inequality holds as $h_1h_4 \geq 400 h_5s$. Hence, we never delete more than $t-s \geq t/2$ webs, concluding that the above process ends with $s$ pairwise connected webs, providing a $K_s$-subdivision.
	\end{proof}

\section{Top layer $(\eps,\eps d)$-robust-expander}\label{sec_ed_expander}
In this section, we prove Lemma~\ref{lem_dense}. Suppose $0< \frac{1}{n} \ll \eta \ll \frac{1}{T} \ll \eps$ and let $K = \frac{Tn}{d}$. Let $G$ be an $n$-vertex $(\eps, \eps d)$-robust-expander with $\delta(G) \geq d/4$ and $d(G) \geq d/2$.
We aim to find a $K_t$-subdivision in $G$ with $t=\min\left\{\frac{d}{(\log d)^{60}}, \frac{\sqrt{n}}{(\log \left(\frac{Tn}{d}\right))^{11}}\right\}.$

We will construct $t$ internally disjoint units, and connect them by paths to obtain a desired subdivision. 
In order to build many internally disjoint units, we need to find many vertex-disjoint copies of $K_{1,h_2}$ after setting aside the interior vertices of previously built units. The
following lemma helps us to find stars in $G$ even after deleting a small set of arbitrary vertices.

\begin{claim}\label{disj_stars}
	Let $1 \leq x \leq   10^3 d/\log^3 K$ be an integer and $q$ be an integer that satisfies $qx \leq t^2 \log^{18} K$. 
	Then for any vertex set $W\subseteq V(G)$ with $|W|\leq t^2 \log^{18} K$,
	the graph $G-W$ contains at least $q$ vertex-disjoint copies of $K_{1, x}$.
	\end{claim}
	\begin{poc}
	Let $\cS=\{S_1,\dots, S_{q'}\}$ be a maximal collection of vertex-disjoint copies of $K_{1, x}$ in $G-W$ and let $W'=\bigcup_{i\in [q']} V(S_i)$.
	Suppose for a contradiction that $q'<q \leq t^2 \log^{18} K/x$, then
	$$|W\cup W'|\leq t^2 \log^{18} K +  q\cdot (1+ x) \leq 3 t^2 \log^{18} K \leq \frac{\eps n}{\log^{3} K} \leq \frac{1}{20} \rho(n,\eps,\eps d)\cdot n.$$
	Lemma~\ref{lem_avg_deg} then implies that $d(G -(W\cup W'))\ge \frac{\eps d}{20\log^{2} K} \geq x$. Hence $G-(W\cup W')$ contains a vertex with degree at least $x$, contradicting the maximality of $\cS$.
	\end{poc}

	We are now ready to build $t$ internally disjoint units. Let us first describe the idea. Imagine we have already built several units with the set of interior vertices being $W$, and we are currently building additional units with centers $v_1,\dots, v_q$ with the set of interior vertices of these partially-built units being $W'$.
	Each partially-built unit is $(h'_1,h_2,h_3)$-unit for some $h'_1$ smaller than we want. The set $W_i$ will be the vertices of a partially-built unit centered at $v_i$. The above claim provides many stars in $G-(W \cup W')$. We find some $i$ and not-so-large $r$ such that $B_{G-(W \cup W_i)}^r(v_i)$ is large so we can find a path from this to one of the stars while avoiding $W\cup W_i$.
	This would add one more branch to one of the partially-built units yielding an $(h'_1+1,h_2,h_3)$-unit, and repeating this will yield an additional unit as desired. 

	\begin{claim}\label{disj_units}
		The graph $G$ contains $t$ internally disjoint $(10 t, t \log^{7} K, 40\eps^{-1}\log^3 K)$-units.
	\end{claim}
	\begin{poc}
		Let $\mathcal{S}$ be a maximum collection of internally disjoint $( 10t, t \log^{7} K, 40\eps^{-1}\log^3 K)$-units of $G$. Suppose that $|\mathcal{S}| < t$. Let $W'$ be the set of interior vertices of the units in $\mathcal{S}$ and set $q= t^2 \log^{20} K / d$. As each unit has at most $1+ 10 t \cdot 40\eps^{-1}\log^3 K$ interior vertices, 
		$$|W'| \leq \eps t\cdot (1+ 10 t \cdot 40\eps^{-1}\log^3 K)  \leq  t^2 \log^{4} K.$$
		Claim~\ref{disj_stars} together with the fact that $t (\log K)^{10} \leq \frac{d}{(\log d)^{60}} \cdot (\log n)^{10} \leq d$ implies that $G-W'$ contains at least $100t$ vertex-disjoint copies of $K_{1, 2t \log^7 K}$. 
		Let $ R_1, \cdots, R_{100t}$ be vertex-disjoint copies of $K_{1,2t\log^{7}{K}}$ in $G-W'$ and let $W = W'\cup \bigcup_{i\in [100t]} V(R_i)$.
		Then we have
		\begin{align}\label{eq: W size}
		|W| \leq |W'|+ 100t(1+2t\log^7 K) \leq t^2 \log^{8} K.
		\end{align}
		Apply Claim~\ref{disj_stars} once again, there exist at least $q$ vertex-disjoint copies of $K_{1, 2d/\log^3 K}$ in $G-W$. 
		Let $S_1, \cdots, S_{q}$ be vertex-disjoint copies of $K_{1, 2d/\log^3 K}$ in $G-W$.

		We aim to build $10 t$ paths from $S_i$ for some $i\in [q]$ to distinct $R_j$ to obtain a unit.
		\begin{subclaim}\label{connects_units}
			There exists $i \in [q]$ and  
			$10t$ vertex-disjoint paths $P_1,\dots, P_{10 t}$ each of length at most $30\eps^{-1}\log^3 K$ in $G-W'$ such that $P_\ell$ is a path between $S_i$ and $R_{j_\ell}$.
			Moreover, $j_1,\dots, j_{10 t}$ are all distinct.
		\end{subclaim}
		\begin{proof}[Proof of subclaim]

			Suppose that such $i\in [q]$ does not exist. 
			For each $i\in [q]$, let $\mathcal{P}_i$ be a maximal collection of vertex-disjoint paths of length at most $30 \eps^{-1} \log^3 K$ from $S_i$ to $R_j$ where no two paths in the collection connect $S_i$ to $R_j$ for the same $j$.
			Let 
			$$ W_i = \bigcup_{P\in \mathcal{P}_i} V(P) \enspace \text{and} \enspace A_i = V(S_i)\setminus W_i.$$
			Note that for  distinct $i,i'\in [q]$, the paths in $\mathcal{P}_i$ and $\mathcal{P}_{i'}$ are not necessarily disjoint. For each $i\in [q]$, as each path has length at most $30 \eps^{-1} \log^3 K$ and $|\mathcal{P}_i| < 10t$ by the assumption, we have 
			\begin{align}
			|W_i| \leq 300\eps^{-1}t\log^3K \enspace \text{ and } \enspace  |A_i|\geq |S_i| - |W_i| \geq d/\log^3 K.\label{eq: Ai size} 
			\end{align}

            We now wish to apply Corollary~\ref{cor_single_large_set} with the following parameters.\newline

            \noindent
{
\begin{tabular}{c|c|c|c|c|c|c|c|c|c}
objects/parameter &  $[\eps d, \frac{n}{2}]$ & $\frac{\eps}{\log^2 K}$ & $q/\log^4 K$  & $d/\log^3 K$ & $q$ & $15\eps^{-1} \log^3 K$ & $W$ & $W_i$ & $A_i$ \\ \hline
playing the role of & $[a,b]$ & $\rho$ & $k$ & $x$ &   $q$ & $r$ & $W$ & $W_i$ & $A_i$ 
\end{tabular}
}\newline \vspace{0.2cm}
   
        For this application, we need to check several conditions in Lemma~\ref{many_large_sets} and Corollary~\ref{cor_single_large_set}. First, $G$ has an $(\eps d, \frac{n}{2}, \frac{\eps}{\log^2 K})$-expansion-property. 
		The conditions $q>20k 
		\log k/\rho$, $a \leq kx \leq \frac{b}{20}$, and $r \geq 10 \log k/\rho$ can be verified as follows: $q  > 20 \cdot \frac{q}{\log^4 K} \cdot \log K \cdot (\frac{\eps}{\log^2 K})^{-1}$, and $\eps d \leq \frac{t^2 \log^{16} K }{d} \cdot \frac{d}{\log^3 K} \leq \frac{n}{40}$, and $r = 15\eps^{-1} \log^3 K \geq 10 \cdot \log q \cdot (\frac{\eps}{\log^2 K})^{-1}$,
		where the first and last inequalities follow from $\log q \leq \log \frac{n}{d} \leq \log K$. Also, the conditions in Lemma~\ref{many_large_sets} hold as $|W| \leq t^2 \log^8 K \leq \frac{1}{2} \cdot \frac{\eps }{\log^2 K} \cdot\frac{t^2 \log^{16} K }{d} \cdot \frac{d}{\log^3 K}$, and $|W_i| \leq 300 \eps^{-1} t\log^3 K \leq  \frac{1}{4} \cdot \frac{\eps}{\log^2K} \cdot \frac{d}{\log^3 K}$. 
        Conditions \emref{mandu1}, \emref{mandu2} an \emref{collectively_large_B} follow immediately from the definition of $W,W_i$ and $A_i$.
        Hence, Corollary~\ref{cor_single_large_set} yields that there exists $i\in [q]$ such that $|B^{r}_{G- (W\cup W_i)}(V(S_i))| \geq  \frac{1}{2}t^2 \log^{13} K.$
			
			By the assumption, the paths in $\mathcal{P}_i$ connect $S_i$ to $R_j$ for fewer than $10t$ distinct $j\in [100 t]$.
			Let $K\subseteq [100t]$ be the set of indices $j$ such that no path in $\mathcal{P}_i$ connects $S_i$ and $R_j$, then $|K|\geq 90t$.
			Let $R_{K} = \bigcup_{j \in K} V(R_i)$, then 
			$|R_K| \geq 90t\cdot t \log^{7} K \geq t^2 \log^{7} K.$
			On the other hand,
			$$|W'| + |W_i| \leq t^2 \log^4 K+ 300 \eps^{-1} t \log^{3}{K}  \leq  \rho(t^2 \log^{7} K)\cdot \frac{1}{4}t^2 \log^{7} K.$$ 
			Thus we can apply Lemma~\ref{lem: short_path_general} to obtain a path of length at most $\frac{10}{\eps} \log^2 K \cdot \log(n/x) \le 10\eps^{-1} \log^3 K$ that connects $R_K$ and $B^{r}_{G- (W\cup W_i)}(V(S_i))$ disjoint from $W'\cup W_i$.
			As $r \leq 15 \eps^{-1}\log^3 K$, this yields a path from $S_i$ to $R_K$ with length at most $10\eps^{-1} \log^3 K + r \leq 30\eps^{-1} \log^{3} K$ which avoids any vertices in $W'\cup W_i$. Hence adding this path to $\mathcal{P}_i$ contradicts the maximality of $\mathcal{P}_i$.
			Therefore, there exists $i \in [q]$ such that
			$|\mathcal{P}_i|\geq 10t$. This finishes the proof of the subclaim.
		\end{proof}
		
		Let $i\in [q]$ and $\mathcal{P}_i$ be the set of $10t$ vertex-disjoint paths obtained from Subclaim~\ref{connects_units}.
		Extending each path in $\cP_i$ to the center of $S_i$, say $v$, and the center of each $R_j$, we have a vertex $v$ and $10t$ disjoint stars $R_j$ and internally disjoint paths connecting $v$ and each star. Note that we increase the size of each path by at most $2$ so the length of the path is at most $40 \eps^{-1} \log^3 K$.
		
		Some leaves in the stars may belong to a path in $\mathcal{P}_i$. However, there are at most $10t \cdot 40 \eps^{-1} \log^3 K < t \log^{5} K$ vertices of the paths in $\mathcal{P}_i$, so each star $R_j$ still have at least $2t \log^{7} K - t\log^{7} K \geq t\log^{7} n$ leaves remaining.  Hence, these paths and the stars all together form a $(10t, t \log^{7} K, 40\eps^{-1}\log^{3} K)$-unit, contradicting the maximality of $\mathcal{S}$.
		Therefore, we have at least $t$ internally disjoint $(10t, t \log^{7} K, 40\eps^{-1}\log^3 K)$-units in $G$, finishing the proof of Claim~\ref{disj_units}.
		\end{poc}

	Let $U_1, \cdots, U_t$ be internally disjoint $(10t, t \log^{7} K, 40\eps^{-1}\log^{3} K)$-units we obtain from Claim~\ref{disj_units}. 
	We now use these internally disjoint units to construct a subdivision of $K_t$. 	
	For each $i\in [t]$, let $v_i$ be the center of $U_i$ and let $W$ be the set of interior vertices of $U_1, \cdots, U_t$. We run the following Algorithm \ref{algo_1}.

	Note that in Line 7, we only delete one exterior star of $U_i$ and $U_j$, because $P_{ij}$ was taken to be a shortest path between $U_i$ and $U_j$ in Line $4$.
    We will inductively prove that $|V(P_{ij})| \leq 100\log^{4} n$ for every $1\leq i<j\leq t$.
	Suppose that $|V(P_{ij}) | \leq 100\eps^{-1}\log^{3} K$ holds for all previous iterations so that $P$ has size at most ${t \choose 2} \cdot 100 \eps^{-1}\log^3 K$. 
	For each $\ell \in [q]$, at most $t$ exterior stars are deleted on Line 7 of the previous iterations. Some stars are deleted on Line 9 of the previous iterations, but this happens only when $\frac{1}{2}t\log^{7} K$ leaves are in $P$. Hence, we have deleted at most 
	$t+\frac{|P|}{\frac{1}{2}t \log^7 K} \leq 2t $ stars of $U_\ell$.

 \begin{algorithm}[H]
		\caption{ }
		\label{algo_1}
		\begin{algorithmic}[1]
			\State $P=\varnothing$.
			
			\For {$i=1, 2, \cdots, t$}
			\For {$j=i+1, i+2, \cdots, t$}
			\State Connect exteriors of $U_i$ and $U_j$ by a shortest path $P_{ij}$  while avoiding $P$ and $W$. \label{algo_line_3}
			\State Extend $P_{ij}$ along the paths inside of $U_i$ and $U_j$ so that $P_{ij}$ connects $v_i$ and $v_j$.
			\State $P \gets P \cup V(P_{ij})$.
			\State Delete the exterior stars of $U_i$ and $U_j$ that intersect with $P_{ij}$.
			\For {$k=1, 2, \cdots, t$}
			\State If an exterior star of $U_k$ has at least $\frac{1}{2} t \log^7 K$ leaves  contained in $P$, then delete the exterior star from $U_k$.
			\EndFor
			\EndFor
			\EndFor
		\end{algorithmic}
	\end{algorithm}
	
	Hence both $U_i$ and $U_j$ still contain at least $8t$ exterior stars and each of them has at least $\frac{1}{2}t\log^7 K$ leaves in $G - P$.
	Thus, $U_i$ and $U_j$ both have at least $4t^2 \log^7 K$ vertices in their exterior disjoint from $P$. 	We have
	$$|P \cup W| \leq {t \choose 2} \cdot 100\eps^{-1}\log^{3} K + t(1+ 10t\cdot 40 \eps^{-1} \log^{3} K ) \leq t^2\log^{4} K\leq t^2 \log^7 K \cdot \rho(4t^2 \log^7 K).$$
	Hence Lemma~\ref{lem: short_path_general} implies that there exists a path of length at most $\frac{10}{\eps} \log^3 K$ connecting exteriors of $U_i$ and $U_j$ while avoiding $W$ and $P$.
	Therefore, we can always find a path $P_{ij}$ of length at most $10\eps^{-1} \log^3 K$ in Line \ref{algo_line_3} of the algorithm. When we extend $P_{ij}$ to $v_i$ and $v_j$, we only use interior vertices of $U_i$ and $U_j$, thus  $|V(P_{ij})| \leq 10\eps^{-1} \log^3 K+2(1+40 \eps^{-1}\log^3 K)\le 100\eps^{-1}\log^3 K$.
	
	When the algorithm terminates, we have a path $P_{ij}$ connecting $v_i$ and $v_j$ for all $1 \leq i < j \leq t$, and by the construction, they are internally disjoint.
	Therefore, $G$ contains a subdivision of $K_t$. This finishes the proof of Lemma~\ref{lem_dense}.

	
\section{Middle layer $(\eps,d^2)$-robust-expander}\label{sec_d2_expander}
In this section, we will prove Lemma~\ref{lem_d2_expander}. 
Let $0< 1/n \ll b \ll \eps <1$ and $d$ be as given in the lemma.
Assume that $G$ does not have a $K_{bd}$-subdivision as we are done otherwise. 
As $\eps d  \leq d^2 \leq d^2 \log^9 n$, Lemma~\ref{lem_bounded_maxdeg} applies, and yields a subgraph of $G$ with at least $n/2$ veritces. For simplicity, we denote this subgraph by $G$. By reducing the value of $n$ by at most half, we let $n =|G|$. Note that we still have $1/n \ll b$.
So assume that we have an $n$-vertex
 $(\eps/2,d^2)$-robust expander-subgraph $G$ satisfying
\begin{align}\label{eq: d range in d2 expander}
    \frac{d}{8} \leq \delta(G)\leq d(G), \enspace \Delta(G)\leq 20 d^2 (\log n)^{10}, \text{and} \enspace \exp(\log^{1/6}n) \leq d\leq \frac{2\sqrt{n}}{(\log n)^{1000}}.
\end{align}
Let $ h_1 = 2 (\log n)^{10}, \enspace h_2 = \frac{d}{\log^3 n}, \enspace  h_3 = \log^4 n, \enspace h_4= \frac{bd}{(\log\log d)^4}  \enspace \text{ and } \enspace h_5= 50 \log^4 n.$
We aim to build $bd$ internally disjoint $(h_4, h_5, h_1/2, h_2, h_3)$-webs and use them to build a $K_{\Omega(h_4)}$-subdivision. For this, we first find many stars and use them to construct units.
\begin{claim}\label{cl: stars in d2 expander}
	For any vertex set $Q \subseteq V(G)$ with $|Q|\leq \frac{n}{\log^5 n}$,
	the graph $G-Q$ has at least $\frac{n}{d \log^5 n}$ vertex-disjoint copies of $K_{1,4h_2}$.
\end{claim}

\begin{claim}\label{cl: units in d2 expander}
	For any vertex set $W \subseteq V(G)$ with $|W|\leq \frac{n}{(\log n)^{100}}$, the graph $G-W$ contains at least $\frac{n}{d(\log n)^{150}}$ vertex-disjoint $(h_1, 2 h_2, h_3)$-units.
\end{claim}

The proof of Claim~\ref{cl: stars in d2 expander} is similar to Claim~\ref{disj_stars}. The proof of Claim~\ref{cl: units in d2 expander} amounts to checking that the technical conditions in~Lemma~\ref{lem_units} are satisfied and hence it can be applied to obtain the desired units. As these proofs are not illuminating, we include them in the arXiv appendix.

Recall that the main issue in this lemma was to overcome the degree constraint.
Claim~\ref{cl: units in d2 expander} provides many vertex-disjoint units disjoint from a given not-too-large vertex set $W$. Using this, we aim to build many internally disjoint webs.
However, in order to obtain $(h_4, h_5, h_1, h_2, h_3)$-webs, we need stars of size $\Omega(d)$ which Claim~\ref{cl: stars in d2 expander} does not provide.
We consider two cases depending on whether there exists a set $Q\subseteq V(G)$ making $G-Q$ very sparse, and obtain large stars in different ways. 

\smallskip

\noindent\textbf{Case 1. Skewed case:} there exists $Q$ with $|Q| \leq \frac{n}{(\log n)^{50}}$ such that $d(G-Q) < \frac{d}{100}$. Fix such a set $Q$.
Recall that $G$ has the maximum degree at most $20 d^2 (\log n)^{10}$.
For this case, let $$x = \frac{1}{2}d (\log n)^{25}, \enspace p=\frac{d}{200}, \enspace \text{and } \rho_1=\rho_2=\frac{\eps}{\log^2(15 n)}.$$
As many edges are incident to $Q$, we can find many vertex-disjoint stars $K_{1,2x}$ centered at $Q$.

\begin{claim}\label{heavy_q}
	The graph $G$ contains $\frac{d}{100}$ vertex-disjoint copies of $K_{1,2x}$ whose centers are in $Q$ and whose leaves are in $V(G) \setminus Q$.
\end{claim}
\begin{poc}
Let $S_1, \cdots, S_m$ be a maximal collection of vertex-disjoint copies of $K_{1,2x}$ whose centers are in $Q$ and leaves are in $L = V(G)\setminus Q$. 
Let $S = \bigcup_{i\in [m]} V(S_i)$.
Assume that $m<\frac{d}{100}$, as otherwise we are done.
We have $|Q|+|S| \leq \frac{n}{(\log n)^{50}} + 4xm \leq n/10$ by \eqref{eq: d range in d2 expander}, so we have
$|L| \geq n/2$, and $|L\setminus S| \geq |L|/2.$
As we have $\frac{d}{8}\leq \delta(G) \leq \Delta(G)\leq 20 d^2(\log n)^{10}$ from \eqref{eq: d range in d2 expander} and $d(G[L])< \frac{d}{100}$, the number of edges between $Q\setminus S$ and $L\setminus S$ is at least 
\begin{align*}
    e(Q\setminus S, L\setminus S) & \geq 
 \sum_{v\in L\setminus S} d(v) - 2e(G[L]) - \sum_{v\in S\cap Q} d(v) \geq \frac{1}{2}|L| \cdot \frac{d}{8} - \frac{d}{100}|L| - m\Delta(G) \\
&\geq \frac{d}{30}|L|  - \frac{20}{100}d^3 \log{n} \geq |Q| \cdot d (\log n)^{25}.
\end{align*}
The final inequality holds as $|L|\geq n/2$ and \eqref{eq: d range in d2 expander} imply that $n \geq \frac{1}{4}d^2 (\log n)^{2000}$.
Thus, there exists a vertex $v \in Q \setminus S$ which has at least $d \log^{25} n$ neighbours in $L\setminus S$, condtradicting the maximality of $S_1,\dots, S_m$. This shows that $m \geq \frac{d}{100}$, and proves the claim.
\end{poc}
Now we construct webs as follows.
\begin{claim}\label{claim_skewed_webs}
In Case 1, there exist $\frac{d}{10^4}$ internally vertex-disjoint $(\frac{1}{2}d, h_5, \frac{1}{2}h_1, h_2, h_3)$-webs.
\end{claim}
\begin{poc}
	Let $S_1,\dots, S_{d/100}$ be vertex-disjoint copies of $K_{1,2x}$, obtained by Claim~\ref{heavy_q}. For each $i\in [d/100]$, let $v_i$ be the center of $S_{i}$ and let $C=\{v_1, v_2, \cdots, v_{d/100}\}$. 
 	Choose a maximal collection of $(d/2,h_5, h_1/2, h_2,h_3)$-webs $U_1, U_2, \cdots, U_s$ such that each $U_i$ does not contain any $v_j$ except at its center and the exterior of $U_i$ does not intersect with any $S_j$ with $j\in [d/100]$. 
 	Assume that $s<\frac{d}{10^4}$.
	Consider the stars $S_i$ whose center does not belong to any of $U_1, \cdots, U_s$. Let $I_0$ be the set of indices of such stars. 
 	In other words, let 
 	$$I_0 =  \{ i\in [d/100]: v_i \notin U_j \text{ for any } j\in [s]\}.$$
 	Then the choice of $U_1, \dots, U_s$ ensures that $|I_0|= \frac{d}{100}-s \geq \frac{d}{150}$.
 	
	Let $B$ be a set of the interior vertices of $U_1, \dots, U_s$. Then,
	$$|B| \leq \frac{d}{10^4}\Big(1+ \frac{d}{2}\cdot h_5 + \frac{d}{2} \cdot \frac{h_1h_3}{2}\Big) \leq \frac{d}{10^4}\left(1+ d(\log^5 n + (\log n)^{14}) \right) \leq d^2 (\log n)^{14}.$$
	Now, let $I=\{ i\in I_0: |S_i\cap B|\leq x \}.$
	As $x= \frac{1}{2}d(\log n)^{25}$, we have $|I| \geq |I_0|  - \frac{d^2 (\log n)^{14}}{x} \geq \frac{d}{200}= p.$
	As $B$ does not contain any center vertices of the stars $S_i$ for $i \in I_0$, this provides $p$ vertex-disjoint stars $S^1, \dots, S^{p}$ in $G-B$ with $\{S^1,\dots, S^{p}\}\subseteq \{S_1,\dots, S_{d/100}\}$. Moreover, letting $C'$ be the set of the center vertices of $S^1,\dots, S^p$, we have $C'\subseteq C$. 
	Also, letting $S = \bigcup_{i\in [d/100]}V(S_i)$, we have
	$|S| \leq \frac{d}{100}\cdot (1+2x) \leq d^2 \log^{25} n\leq \frac{n}{d (\log n)^{150}}.$
	Thus, by Claim~\ref{cl: units in d2 expander}, the graph $G-B-S$ has at least $q=d \log^{20} n \leq \frac{n}{d(\log n)^{150}}$ vertex-disjoint $(h_1,2h_2, h_3)$-units.
	Hence, $G-B - (C \setminus C')$ contains $p$ stars with $x$ leaves and $q$ distinct $(h_1, h_2, h_3)$-units in such a way that all stars and units are pairwise vertex-disjoint.
	
	We wish to apply Lemma~\ref{lem_webs} with the parameters $p,q,x, h_2, h_3, h_5$ and $\rho_1$, $\rho_2$ playing roles of themselves and $h_1/2$, $d/2$, and $B \cup (C \setminus C')$ playing roles of $h_1$, $h_4$, and $W$, respectively. Let us check all the conditions. As $G$ is an $(\eps/2,d^2)$-expander and $d^2\leq px= \frac{1}{400}d^2(\log n)^{25} < \frac{n}{100}$, it has $(px,4px,\rho_1)$-expansion-property as well as $(px,n/2,\rho_2)$-expansion-property.
Moreover, we have $h_3 \leq \log^4 n$, $\frac{2n}{px} \leq n$, and $\frac{2n}{qh_1h_2} \leq n$. Hence, $\frac{1}{2} h_1 = (\log n)^{10} \geq 4 h_5$, $\frac{d}{2} \leq c x \rho_1^2$, $q = d(\log n)^{20} \geq 4 \cdot \left(\frac{d}{2}\right)$, and $h_5 = 50 \log^4 n \geq h_3 + 50\rho_1^{-1}\log q + 10\rho_2^{-1} \log\left(\frac{2n}{px}\right) + 10 \rho_2^{-1}\log\left( \frac{2n}{qh_1h_2}\right)$. Finally, as $px= \frac{1}{400} d^2 (\log n)^{25}, qh_1h_2 = 2d^2 (\log n)^{27}$ and $(\frac{d}{2}) h_5 \leq  25 d \log^4 n$, we have $|B \cup (C \setminus C')|+p \leq d^2 (\log n)^{14} + \frac{d}{100}+p$, which is at most $\frac{1}{32}\rho_1 p x$ and also at most $\min\{\frac{1}{16} qh_1h_2 - dh_5,\frac{1}{8} \rho px -  dh_5\}$.
Therefore, we can apply Lemma~\ref{lem_webs} to obtain a $(\frac{1}{2}d, h_5, \frac{1}{2}h_1, h_2, h_3)$-web $U_{s+1}$.
 As $U_{s+1}$ avoids all centers of $S^1, \dots, S^p$ except one and disjoint from $C \setminus C'$, it does not contain any vertex in $C$ except at its center.
 Also, the exterior of $U_{s+1}$ belongs to the units disjoint from $S$, it does not intersect with any stars $S_i$.
 This contradicts the maximality of $S$, proving the claim.
 \end{poc}

\noindent \textbf{Case 2. Uniform Case}: for every $Q \subseteq V(G)$ with $|Q| \leq \frac{n}{\log^{50} n}$, we have $d(G-Q) \geq \frac{d}{100}$. In this case, repeatedly taking stars and applying the condition to the remaining part of the graphs, we can obtain many disjoint copies of $K_{1,d/100}$ to build many webs.
\begin{claim}\label{claim_uniform_webs}
    In Case 2, the graph $G$ contains $bd$ internally disjoint $(h_4, h_5,\frac{1}{2}h_1, h_2, h_3)$-webs.
\end{claim}
Claim~\ref{claim_uniform_webs} is simpler than Claim~\ref{claim_skewed_webs}: we just need to check that Lemma~\ref{lem_webs} applies. We include its proof in the arXiv appendix.

By Claims~\ref{claim_skewed_webs} and~\ref{claim_uniform_webs}, in both cases, we obtain at least $bd \geq \frac{d}{(\log \log n)^4}$ internally disjoint $(h_4, h_5, \frac{1}{2}h_1, h_2, h_3)$-webs.
Finally, we connect these webs to build a desired subdivision using Lemma~\ref{lem_conn_webs}.
Let $t = \frac{d}{(\log\log d)^4}$,  $s = \frac{bd}{(\log\log d)^4}$, and $\rho = \frac{\eps}{\log^2(15 n)}.$
As $h_4h_5t = \frac{50bd^2 \log^4 n}{(\log\log d)^8}$ and $h_5st = \frac{b d^2 \log^4 n}{(\log\log d)^8}$, we have $h_5 = 50 \log^4 n \geq \log^4 n + \frac{\log n}{\rho} = h_3 +\frac{1}{\rho}\log(\frac{8n}{h_1h_2h_4} )$, $\frac{1}{2}\rho h_1 h_2h_4  \geq \frac{\eps d^2 \log^5 n}{(\log\log d)^4} \geq d^2 \log^4 n \geq 50(h_4h_5 t+ 20 h_5 st)$, and $\frac{1}{2} h_1h_4  = \frac{ bd(\log n)^{10}}{(\log\log d)^{4}} \geq \frac{20000bd \log^4 n}{(\log \log d)^4}  =  400 h_5s$. Therefore, we can apply Lemma~\ref{lem_conn_webs} to obtain a $K_{bd/(\log \log d)^4}$-subdivision.

\section{Bottom layer $(\eps,\frac{1}{100}c(G))$-robust expander}\label{sec_nc_expander}
In this section, we streamline the proof Lemma~\ref{lem_nc_expander}, highlighting only the difference with the proof of Lemma~\ref{lem_d2_expander}, and leave the detailed proofs of some similar claims to online appendix. Let $G$ and $H$ be graphs given as in the Lemma~\ref{lem_nc_expander}.
For brevity, we write $n_H=|H|$ and $n_c= c(G)$. 
Choose a number $T$ and $\eta$ such that $0<\frac{1}{n} \ll b \ll \eta \ll \frac{1}{T}\ll \eps\ll 1$.
Let $K = T \cdot \frac{n_H}{n_c}$, $t=\min\Big\{d, \frac{\sqrt{n_c}}{\sqrt{\log n_c}}\Big\}$, and $\rho=\frac{\eps}{\log^2 K}$.

By the assumption of the lemma, $d\geq \frac{\sqrt{n_H}}{(\log n_H)^{1000}}$ and $n_c \geq \frac{n_H}{(\log n_H)^{3000}}$. As $H$ is a subgraph with average degree at least $d/100$, by the definition of $c_{1/100}(G)$, we have $n_H \geq n_c$.
Thus, 
\begin{align}\label{eq: log nc and log nH}
 t\geq \min\left\{\frac{\sqrt{n_H}}{(\log n_H)^{1001}}, \frac{\sqrt{n_H}}{(\log n_H)^{1501}} \right\} \geq \frac{\sqrt{n_H}}{(\log n_H)^{2000}} \enspace \text{and} \enspace \frac{1}{2} \log n_H \leq \log n_c \leq \log n_H.
\end{align}
This implies that we have 
\begin{align}\label{eq: log nHt2}
\log K = \log T + \log\left(\frac{n_H}{n_c}\right) \leq 4000 \log\log n_c \text{ and } \log\left(\frac{n_H}{t^2}\right) \leq 4000 \log\log n_H \leq 5000 \log\log n_c.
\end{align}

We will construct webs of certain size to obtain a subdivision of $K_{\Omega(t/(\log \log n_c)^{6})}$.
For this, we begin with the following important observation.
\begin{claim}\label{cl: expansion in c100}
The graph $H$ has the $(1, \frac{n_H}{2}, \rho)$-expansion-property.
\end{claim}
\begin{poc}
    Consider a vertex set $X\subseteq V(H)$.
    We claim that if $|X|\leq \frac{n_c}{100}$, then $|N_H(X)|\geq |X|$.
    Indeed, for such a vertex set $X$, the average degree of the graph $G[X\cup N_H(X)]$ is
    $$d(G[X\cup N_H(X)]) \geq \frac{\delta(H)|X| }{|X\cup N_H(X)|} \geq \frac{d|X|/16}{|X\cup N_H(X)|}.$$
    If $|N_H(X)|<|X|$, then this is at least $\frac{d}{32}$. However, $|X\cup N_H(X)|< n_c$, a contradiction to the definition of $n_c=c_{1/100}(G)$.
    
    If $\frac{n_c}{100}\leq |X|\leq \frac{n_H}{2}$, then as $H$ is an $(\eps,\frac{n_c}{100})$-robust-expander, we have the desired expansion
    $|N_H(X)|\geq  \frac{\eps}{\log^2(\frac{1500|X|}{n_c} )} |X| \geq \frac{\eps}{\log^2 K}|X|.$
\end{poc}

Recall that we have $t^2 \log n_c \leq n_c$.
As $0<1/T\ll 1$, we have $\log^{1000}\left(Tx\right) \leq Tx$ for any $x\geq 1$.
Then 
\begin{align}\label{eq: t2lognc}
    t^2 \log n_c \cdot \log^{1000}K \leq n_c \cdot \log^{1000}\Big(\frac{Tn_H}{n_c}\Big) \leq T n_H.
\end{align}
Let $h_1 = \log n_c (\log\log n_c) (\log{K})^{20}$,  $h_2 = \frac{\eta^{1/2} t}{\log^2 K}$, $h_3 = \log\log n_c (\log K)^{10}$, $h_4 = \frac{\eta t}{(\log\log n_c)^6}$, and $h_5 = 20 \log n_c (\log K)^{10}$, further set $p=t \log n_c (\log K)^{100}$, $q=\frac{t (\log K)^{10}}{(\log \log n_c)^4}$, $k=\frac{10 t \log n_c (\log K)^{45}}{(\log \log n_c)^2}$, and $x=2h_2$.

With the help of Claim~\ref{cl: expansion in c100}, we can carry out a similar approach as Lemma~\ref{lem_d2_expander} to construct the webs needed for the desired subdivision.

\begin{claim}\label{cl: stars in c100 expander}
	For any vertex set $Q \subseteq V(H)$ with $|Q|\leq \frac{\eps n_H}{100\log^4 K}$,
	the graph $H-Q$ has at least $2t  \log n_c (\log K)^{100}$ vertex-disjoint copies of $K_{1,T h_2}$.
\end{claim}

\begin{claim}\label{cl: units in c100 expander}
    For any vertex set $W\subseteq V(H)$, if
    $|W|\leq \frac{t^2 \log n_c (\log K)^{30}}{(\log\log n_c)^3},$
    then $H-W$ contains $q$ vertex-disjoint $(2h_1,h_2,h_3)$-units.
\end{claim}

\begin{claim}\label{cl: webs in c100 expander}
    $H$ contains $h_4$ internally disjoint $(h_4,h_5,h_1,h_2,h_3)$-webs.
\end{claim}

Finally, we will use the webs given by Claim~\ref{cl: webs in c100 expander} to build a subdivision of $K_{\Omega(t/(\log \log n_c)^6)}$.
Let $t'=h_4$ and $s'=\frac{t'}{30}.$
We aim to apply Lemma~\ref{lem_conn_webs} with parameters $n_H, t', h_1, h_2, h_3, h_4, h_5, s'$ and $\rho$.
By our choice of parameters, all these numbers except $\rho$ are larger than $T$.
We now check the conditions of Lemma~\ref{lem_conn_webs}.
As $\frac{8n_H}{h_1h_2h_4} \leq \frac{n_H}{t^2} (\log\log n_c)^6$, we have 
$$h_5  = 20 \log n_c (\log K)^{10} \stackrel{\eqref{eq: log nHt2}}{\geq} 
    \log n_c + \frac{1}{\rho}\log\left(\frac{n_H}{t^2}\right) + \frac{6}{\rho}\log\log\log n_c
    \geq h_3 + \frac{1}{\rho} \log\left( \frac{8n}{h_1h_2h_4}\right),$$ 
    and $\rho h_1h_2h_4 \geq 
    \frac{\eta^{3/2} t^2 \log n_c (\log K)^{15}}{(\log\log n_c)^5} 
    \geq 50 (h_4h_5 t'+ 20h_5s't')$.
Also, we have 
$$h_1h_4 = \frac{\eta t \log n_c (\log K)^{20}}{(\log\log n_c)^5} \geq 
\frac{8000 t \log n_c (\log K)^{10}}{(\log\log n_c)^6 }
\geq 400 h_5 s'.$$
By Claim~\ref{cl: expansion in c100}, $H$ has $(h_1h_2h_4/8,n_H/2,\rho)$-expansion property. 
Therefore, by Lemma~\ref{lem_conn_webs}, we can conclude that $H$ contains a $K_{s}$-subdivision.
By \eqref{eq: log nc and log nH}, we have $\log \log d = \Theta(\log \log n_c)$.
Therefore, we finish the proof of Lemma~\ref{lem_nc_expander}.

	\section{Concluding Remark}\label{sec_conclusion}
	One obvious open problem is to determine the correct exponents on the polylogarithmic factor in Theorem~\ref{main_thm}. It would be also very interesting to give more examples illustrating the replacing average degree by crux paradigm.
	    
	    Motivated by this paradigm and the fact that crux size is a lower bound for the largest sublinear expander subgraph, it is then of great interest to find natural sufficient conditions for a graph $G$ guaranteeing $c(G) = \omega(d(G))$. More generally, when does a graph contain a large crux/subexpander? There are several recent research along this direction: (1)~Krivelevich~\cite{Krivelevich2018} showed that if a graph is locally sparse, then it contains a linear-size subexpander; (2)~Louf and Skerman~\cite{Louf2020} reduced the problem of finding a large expander subgraph to finding a subdivision of an expander; (3)~more recently, Chakraborti, Kim, Kim, Kim and Liu~\cite{Chakraborti2021} proved that if a graph contains a positive fraction of vertices from which the random walk mixes rapidly, then it contains an almost spanning expander subgraph. 

      The notion of crux has one drawback using the minimum order rather than the maximum order over all dense subgraphs. For example, consider disjoint union of $\frac{n}{2t} - \frac{t}{2}$ copies of $K_{t,t}$ and one copy of a random graph $G(t^2, \frac{1}{t})$.
	This has average degree $(1+o(1))t$ and its smallest $\alpha$-crux is a copy of $K_{t,t}$ for $\alpha<1-o(1)$ while the component $G(t^2, \frac{1}{t})$ is more essential part of the graph having a larger clique subdivision.
	
	We can overcome this drawback as follows.
	Consider all subgraphs $G'$ of $G$ with average degree at least, say, $\sqrt{\alpha}d$ and measure their crux size $c_{\sqrt{\alpha}}(G')$.
	We take $G'$ with the maximum $c_{\sqrt{\alpha}}(G')$ and consider $c_{\sqrt{\alpha}}(G')$ as the essential order of $G'$ (and $G$ as well). Then it is easy to check that $c_{\sqrt{\alpha}}(G') \geq c_{\alpha}(G)$.
	In fact, this is a strict inequality for the above example of the disjoint union of $\frac{n}{2t} - \frac{t}{2}$ copies of $K_{t,t}$ and one copy of a random graph $G(t^2, 1/t)$. We know that $c_{\sqrt{\alpha}}(G') = t^2$ while $c_{\alpha}(G)=t$ for $\alpha <1$. So this can pin down the essential part of $G$ more precisely.

    One can of course define 
    $$\displaystyle \max_{G'\subseteq G, d(G')\geq \sqrt{\alpha}d} c_{\sqrt{\alpha}}(G')$$ 
    as the `essential order' of $G$. However, we choose not to do this for the sake of brevity. We can simply assume that our graph $G$ is actually the graph $G'$ that we have passed onto from the original graph. Nonetheless, this argument shows that the notion of the crux is not really ruined by local noises, capturing the essential order of the graph.

\section*{Acknowledgement}
We thank the referees for their detailed comments which drastically improve the presentation of this paper.

\setlength{\parskip}{0pt}
\setlength{\itemsep}{0pt plus 0.3ex}

\appendix

\section{Proofs of Lemmas~\ref{lem_sparse} and~\ref{lem_bounded_maxdeg}}
We first prove Lemma~\ref{lem_bounded_maxdeg}.
\begin{proof}[Proofs of Lemma~\ref{lem_bounded_maxdeg}]
		We first claim that the absence of $K_{bd}$-subdivision implies that $G$ contains only a few vertices of large degree. More precisely, 
		\begin{align}\label{eq: large deg vtx}
		\big|\{v \in V(G) :d_G(v) > 10 d^2 (\log n)^{10}\}\big| < bd.
		\end{align}
		Indeed, if \eqref{eq: large deg vtx} does not hold, let $A=\{v_1, \ldots, v_{bd}\}$ be a set of exactly $bd$ vertices of degree larger than $10 d^2 (\log n)^{10}$. We will construct a $K_{bd}$-subdivision by finding internally disjoint paths of length at most $\log^4 n$ between any pair of vertices in $A$. For this, consider the following simple greedy algorithm.
		\begin{algorithm}[H]
			\caption{ }
			\label{algo_2}
			\begin{algorithmic}[1]
				\State $P=\varnothing$, $A=\{v_1,\ldots, v_{bd}\}$.
				\For{$i=1, 2, \cdots, bd$}
				\For{$j=i+1, i+2, \cdots, bd$}
				\State Let $P_{ij}$ be a shortest path connecting $v_i$ and $v_j$ while avoiding $P \cup A \setminus \{v_i, v_j\}$. \label{algo2_path}
				\State $P \gets P \cup V(P_{ij})$.
				\EndFor
				\EndFor
				
			\end{algorithmic}
		\end{algorithm}
		We use induction to prove that we can find a path of length at most $\log^4{n}$ at Line~\ref{algo2_path} at each iteration. If this was the case for all iterations up to $i \leq \ell-1$ for some $\ell \in [bd]$, we have
		$$ |P \cup A | \leq {bd \choose 2}\cdot \log^4 n+ bd \leq \frac{1}{4}  \rho(|N(v_{\ell})|)\cdot |N(v_{\ell})|$$
		as $|N(v_{\ell})| > 10d^2 (\log n)^{10}$ and $\rho(x) \geq 1/\log^3 n$ for every $k/5 \leq x \leq n$. 
		Then Lemma~\ref{short_path} implies that there exists a path between $N(v_i)$ and $N(v_j)$ of length at most $\frac{2}{\eps}\log^3(\frac{15n}{\eps d}) \leq (\log^4 n)-2$ avoiding $A \cup P$. Thus, the shortest path connecting $v_i$ and $v_j$ while avoiding $(A \cup P) \setminus \{v_i, v_j\}$ has length at most  $\log^4 n$. Therefore, by repeating this for all $\ell \in [bd]$, we have internally vertex-disjoint paths connecting vertices of $A$ yielding a copy of $K_{bd}$-subdivision in $G$. This proves \eqref{eq: large deg vtx}.
		
		Now, we delete all vertices of degree larger than $10d^2 \log^{10} n$ from $G$ to obtain a graph $H$. As we have deleted less than $bd< n/2$ vertices, $H$ has at least $n/2$ vertices and $\delta(H) \geq d/2-bd \geq d/4$. 
		Also, by construction, $\Delta(H) \leq 10 d^2 (\log n)^{10}$.
		
		Hence, it only remains to prove that $H$ is  an $(\eps/2, k)$-robust-expander.  
		Consider a set $X \subseteq V(H)$ with $\frac{k}{2} \leq |X| \leq \frac{|V(H)|}{2}$.
		As $\rho(x,\eps/2,k)x$ is an increasing function in $x$, we have 
				\begin{align*}
				   \frac{1}{2} \rho(|X|,\eps/2,k)|X| \geq \frac{1}{2} \rho(k/2,\eps/2,k)\cdot (k/2) \geq \frac{\eps k}{8\log^2(15/2)} \geq bd\geq |A|.
				\end{align*}
		For a set of edges $F$ with $|F| \leq \rho(|X|, \eps/2, k)\cdot d\cdot |X| \leq  \rho(|X|, \eps, k) \cdot d  \cdot |X|$, 
		 we have
		\begin{align*}
		|N_{H\setminus F}(X)| &\geq |N_{G\setminus F}(X)|-|A| 
		\geq \rho(|X|, \eps, k)|X|- |A|  \\
		&\geq \frac{1}{2} \rho(|X|, \eps, k)|X| \geq \rho(|X|, \eps/2, k) |X|.
		\end{align*}
Therefore, $H$ is an $(\eps/2, k)$-robust-expander as desired.
\end{proof}

To prove Lemma~\ref{lem_sparse}, we use the following lemma. This follows from the same proof as Lemma~5.1 in \cite{Liu2017} using slightly different parameters.
	\begin{lemma}[\cite{Liu2017}] \label{lem_sparse_blackbox}
		For $0<\eps< \frac{1}{100}$, there exists $b>0$ such that the following holds. If $H$ is an $n$-vertex $(\eps,\eps d)$-expander with $d\leq \exp(\log^{1/6} n)$, $\delta(H)\geq d/4$ and $\Delta(H)\leq 20 d^2 (\log n)^{10}$, then $H$ contains a $K_{bd}$-subdivision.
	\end{lemma}
	By Lemma~\ref{lem_bounded_maxdeg}, either $G$ has a $K_{bd}$-subdivision, or it contains an $(\eps/2, \eps d)$-expander $H$ as a subgraph with $|H|\geq n/2$ and $\delta(H)\geq d/4$ and $\Delta(H)\leq 10d^2 (\log n)^{10}$. In the latter case, we can invoke Lemma~\ref{lem_sparse_blackbox} to embed $K_{bd}$-subdivision in $H\subseteq G$, finishing the proof of Lemma~\ref{lem_sparse}.

\section{Missing proofs in Lemma~\ref{lem_d2_expander}}

 \begin{proof}[Proof of Claim~\ref{cl: stars in d2 expander}]
    	Consider a maximal collection $\mathcal{S}$ of vertex-disjoint copies of $K_{1,4h_2}$ in $G-Q$. Let $S$ be the set of vertices belonging to one of the stars in $\mathcal{S}$. 
	If $\mathcal{S}$ consists of less than $\frac{n}{d \log^5 n}$ copies, then we have $ |Q\cup S| \leq \frac{n}{\log^5 n } + (1+ 4 h_2)\cdot \frac{n}{d \log^5 n} \leq \frac{n}{\log^4 n}.$
	Hence Lemma \ref{lem_avg_deg} applies and $G-(Q\cup S)$ has average degree at least $\frac{4d}{\log^3 n}=4h_2$. This allows us to add another vertex-disjoint copy of $K_{1,4h_2}$ to the maximal collection, a contradiction.
\end{proof}

\begin{proof}[Proof of Claim~\ref{cl: units in d2 expander}]
Take a maximal collection $\{U_1,\dots, U_q\}$ of vertex-disjoint $(h_1, 2h_2, h_3)$-units in $G-W$. Let $U = \bigcup_{i\in [q]} V(U_i)$. If $q\geq \frac{n}{d(\log n)^{150}}$, then we are done.
Otherwise, we have
 $$|U| \leq \frac{n}{d(\log n)^{150}} \cdot (1+2h_1h_2 + h_1h_3) \leq \frac{n}{(\log n)^{100}}.$$

Note that \eqref{eq: d range in d2 expander} implies that $n \geq d^2 (\log n)^{400}$. Let 
$$ p=\frac{n}{d (\log n)^{20}},\enspace \rho =\frac{\eps}{\log^2(15n)}, \enspace x= 4h_2,  \enspace \text{and} \enspace k= \frac{n}{d(\log{n})^{30}}.$$
Now we wish to apply Lemma~\ref{lem_units} with the parameters $k,x,h_1,h_3, \rho$ playing role of themselves and $2h_2$, $W \cup U$, $1/\eps$ playing roles of $h_2, W, T$, respectively.
As $0<\frac{1}{n} \ll b \ll \eps \ll 1$, \eqref{eq: d range in d2 expander} implies that the integers $n, k, p, x, h_1, 2h_2, h_3$ are all larger than $\frac{1}{\eps}$, which can play the role of $T$.
For this application, we check the conditions of the lemma.
We have $p= (\log n)^{10} k$, $h_1 \leq 2(\log n)^{10}$ and $\log^2 n \rho^{-1}, h_3 \leq \log^5 n$, so the conditions in \eqref{eq: conditions on lem_units1} are satisfied as follows:
  \begin{align*}
   x &= \frac{4d}{\log^3 n} \geq 200 (\log n)^{70} + \frac{x}{2} \geq \frac{200 ph_1h_3}{\rho^2 k} + 2h_2, \\
   p&= k (\log n)^{10} \geq 4kh_3 + 2\rho k h_1.
   \end{align*}
Moreover, as $px = \frac{4n}{(\log n)^{23}}$ and $kx = \frac{4n}{(\log n)^{33}}$, the conditions in \eqref{eq: conditions on lem_units2} also hold as follows:
\begin{align*}
    &4\left( \frac{nh_1}{px} \right)^3 \leq 4 (\log n)^{100} \leq k  \leq \frac{n}{30x}, \\
    h_3&= \log^4 n \geq (\log n)^3 \log( (\log n)^{40} ) \geq \frac{30}{\rho} \log\left(\frac{10 n}{\rho k x} \right).
\end{align*} 
As $G$ is an $(\eps, d^2)$-robust-expander and $d^2 \leq \frac{4n}{(\log n)^{2000}} \leq \frac{kx}{2}$ by \eqref{eq: d range in d2 expander}, $G$ has the $(\frac{kx}{2}, \frac{n}{2}, \rho)$-expansion-property and 
$$|W\cup U| \leq \frac{n}{(\log n)^{100}}+ \frac{n}{(\log n)^{100}} \leq \frac{1}{100} \rho^2 kx.$$
By Claim~\ref{cl: stars in d2 expander}, $G-(W \cup U)$ contains at least $2p$ vertex-disjoint stars of size $x=4h_2$.
Thus we can apply Lemma~\ref{lem_units} to obtain an $(h_1,2h_2,h_3)$-unit in $G-(W\cup U)$. This contradicts the maximality of $\{U_1,\dots, U_q\}$. Thus $G-W$ contains at least $\frac{n}{d(\log n)^{150}}$ vertex-disjoint $(h_1, 2 h_2, h_3)$-units and this finishes the proof of the claim.
\end{proof}

\begin{proof}[Proof of Claim~\ref{claim_uniform_webs}]
	Choose a number $c>0$ such that $0<b \ll c \ll 1$.
Consider a maximal collection of internally disjoint $(h_4, h_5, \frac{1}{2}h_1, h_2, h_3)$-webs. Assume that it contains less than $bd$ webs.
Let $W$ be the set of interior vertices of webs, then we have
$$|W| \leq bd \cdot (1+ h_4h_5 + \frac{1}{2}h_1h_3h_4) \leq d^2 (\log n)^{14}.$$
Let 
$$ x= \frac{d}{100}, \enspace p = 100 d (\log n)^{30} \enspace \text{ and } \enspace q = d(\log n)^{30}.$$
 As we are in Case 2 and $pd+ |W|\leq \frac{n}{(\log n)^{200}}$ by \eqref{eq: d range in d2 expander}, the graph  $G-W$ has average degree at least $d/100$ even after further deleting $p$ vertex-disjoint copies of $K_{1,d/100}$, Hence, we can choose $p$ vertex-disjoint copies $S_1,\dots, S_p$ of $K_{1,d/100}$ in $G-W$. As 
 $$|W\cup \bigcup_{i\in [p]} V(S_i)| \leq d^2 (\log n)^{14} + 2d^2 (\log n)^{30} \leq \frac{n}{(\log n)^{100}}$$
and $q = d(\log n)^{30}$, by Claim~\ref{cl: units in d2 expander}, 
there exist vertex-disjoint $(h_1, 2h_2, h_3)$-units $R_1, R_2, \cdots, R_q$ in $G-W-\bigcup_{i \in [p]} V(S_i)$.
Let $$\rho_1=\frac{\eps}{10^5(\log \log d)^2} \enspace \text{ and } \enspace \rho_2=\frac{\eps}{\log^2(15n)}.$$
We wish to apply Lemma~\ref{lem_webs} with the parameters $p,q,x ,h_2,h_3,h_4, h_5$ and $\rho_1, \rho_2$ and the set $W$ playing roles of themselves and $h_1/2, 1/\eps$ playing role of $h_1, T$, respectively.
As $0< \frac{1}{n} \ll b \ll \eps \ll 1$, \eqref{eq: d range in d2 expander} implies that the numbers $p, q, x, h_1, h_2, h_3, h_4, h_5$ are all larger than $\frac{1}{\eps}$, which can play the role of $T$.

Note that $px = d^2 (\log n)^{30}$.
As $G$ is an $(\eps/2, d^2)$-expander and $px \geq d^2$, it has $(d^2 , \frac{1}{2} n, \rho_2)$-expansion property. Moreover, it has $(px , 4px, \rho_1)$-expansion property as well. Indeed, as $\log d \geq \log^{1/6} n$, we have
$$\rho(4d^2 (\log n)^{30}) = \frac{\eps/2}{\log^2( 60(\log n)^{30} )} \geq \frac{\eps}{2000 (\log\log n)^2} \geq \frac{\eps}{10^5 (\log\log d)^2 }.$$
Hence, $G$ being an $(\eps/2, d^2)$-expander implies that it has $(px , 4px, \rho_1)$-expansion property.
Moreover, we have $h_3 = \log^4 n$, $\frac{2n}{px} \leq n$, and $\frac{2n}{qh_1h_2} \leq n$. Hence, so we have
   \begin{align*}
\frac{1}{2} h_1 &= (\log n)^{10} \geq 4 h_5, \\
h_4 & = \frac{bd}{(\log\log d)^4}  \leq c x \rho_1^2, \\
q &= d(\log n)^{30} \geq 8 h_4, \\
h_5 &= 50 \log^4 n \geq 
h_3 + 50\rho_1^{-1} \log q + 10\rho_2^{-1}  \log\frac{2n}{px} + 10\rho_2^{-1} \log\frac{4n}{qh_1h_2}.
\end{align*}
where the second line comes from the fact that $b\ll \eps$.
Finally, as $px= d^2 (\log n)^{30}, qh_1h_2 = 2d^2 (\log n)^{37}$ and $h_4h_5\leq d \log^4 n$, we have
$$|W| \leq d^2 (\log n)^{14} \leq \left\{
\begin{array}{l}
\frac{1}{32}\rho_1 p x -p,\\
\frac{1}{16} qh_1h_2 - p - 2h_4h_5, \\
\frac{1}{8} \rho px - p  - 2h_4 h_5.\\
\end{array}\right.$$
 Therefore, we can apply Lemma~\ref{lem_webs} to obtain a $(h_4, h_5, h_1/2, h_2, h_3)$-web outside $W$ as desired.
 This contradicts the maximality of our choice, implying that $G$ contains at least $bd$ internally disjoint $(h_4, h_5, h_1/2, h_2, h_3)$-webs.
 This proves the claim.
\end{proof} 

\section{Missing proofs in Lemma~\ref{lem_nc_expander}}
\begin{proof}[Proof of Claim~\ref{cl: stars in c100 expander}]
	Consider a maximal collection of vertex-disjoint copies of $K_{1,Th_2}$ in $H-Q$. Let $S$ be the set of vertices belonging to one of those stars. 
	If the collection of maximal stars consists of less than $2t  \log n_c (\log K)^{100}$ copies, then
	$$ |Q\cup S| \leq \frac{\eps n_H}{100 \log^4 K} + (1+ T h_2)\cdot 2t  \log n_c (\log K)^{100}\leq \frac{1}{100} \rho n_H + \eta^{1/3}  t^2 \log n_c (\log K)^{98}  \leq  \frac{1}{20} \rho n_H.$$ 
	We have the final inequality as $t^2 \log n_c \leq n_c$ and $K = \frac{T n_H}{n_c} > (\log K)^{100}$. 
	Hence Lemma~\ref{lem_avg_deg} applies and $H-(Q\cup S)$ has average degree at least $\frac{\eps d}{20\log^2 K}\geq Th_2$. This allows us to add another vertex-disjoint copy of $K_{1,T h_2}$ to the maximal collection, a contradiction.
\end{proof}

\begin{proof}[Proof of Claim~\ref{cl: units in c100 expander}]
 Let $U_1,\dots, U_m$ be a maximal collection of vertex-disjoint $(2h_1,h_2,h_3)$-units in $H-W$. Assume $m< q$, as otherwise we are done.
 Let $U = \bigcup_{i\in [m]} V(U_i)$.
 Then we have 
 $$|W\cup U|\leq \frac{t^2 \log n_c (\log K)^{30}}{(\log\log n_c)^3} 
 + q( 1+ 2h_1h_3 + 2h_1h_2) \leq \frac{2 t^2 \log n_c (\log K)^{30}}{(\log\log n_c)^3}. $$
 
As $t^2\log n_c\leq n_c$ and $(\log K)^{34} \leq K$, we have 
$|W\cup U| \leq \frac{\eps n_H}{100 \log^4 K}$. So Claim~\ref{cl: stars in c100 expander} implies that $H-(W\cup U)$ contains $2p$ vertex-disjoint copies of $K_{1,2h_2}$.
We aim to apply Lemma~\ref{lem_units} with the parameters $n_H,p,x,k,h_2,h_3$ play the role of themselves and $2h_1$ and $(W \cup U)$ play the roles  of $h_1$ and $W$, respectively.
By our choice of parameters, all those numbers are larger than $T$.
As $kx \leq n_H$, Claim~\ref{cl: expansion in c100} shows that $H$ has the $(\frac{kx}{2}, \frac{n_H}{2}, \rho)$-expansion property.
Note that each of $\frac{p}{k}, h_1,h_3$ and $\rho^{-1}$ are smaller than $\log n_c (\log\log n_c)^2 (\log{K})^{100}$ and $t \geq (n_c)^{1/3}$ by \eqref{eq: log nc and log nH}. Hence 
\begin{align*}
    x&= \frac{\eta^{1/2} t}{\log^2 K} + h_2 \geq 
    200 \left(\log n_c (\log\log n_c)^2 (\log{K})^{100}\right)^5 + h_2 \geq \frac{200ph_1h_3}{\rho^2 k} + h_2, \\
    p&=  t \log n_c (\log K)^{100} \geq \frac{40 t \log n_c (\log K)^{55} }{\log\log n_c} + \frac{20 \eps  t \log n_c (\log K)^{63} }{\log\log n_c}    =
    4kh_3 + 2\rho kh_1.
\end{align*}
This verifies \eqref{eq: conditions on lem_units1}.
Also, as $\frac{10n_H}{\rho kx} \leq  \frac{n_H}{t^2} (\log\log n_c)^3$ and $\frac{n_H}{t^2} \leq (\log n_H)^{4000} $ by \eqref{eq: log nHt2}, we have $\frac{n_Hh_1}{px} \leq \frac{n_H(\log\log n_c)^3}{t^2 \log n_c} \leq (\log \log n_H)^{5000}$ by \eqref{eq: log nc and log nH}, so 
\begin{align*}
    4\left(\frac{n_H h_1}{px}\right)^4 & \leq 4 (\log n_H)^{3\cdot 10^4} \leq  k \leq \frac{n_H}{30 x}, \\ 
    h_3 & = \eps_1^{-1} \log\log n_c (\log K)^{100} \geq 
    30 \eps^{-1} \log^2 K \log\left(\frac{n_H}{t^2}\right) + 90 \eps^{-1} \log^2 K \log\log\log n_c \\
    & \geq 
    \frac{30}{\rho} \log\left(\frac{10n_H}{\rho kx}\right).
\end{align*}
Here, the penultimate inequality holds by \eqref{eq: log nHt2}.
We also have 
$$|W\cup U|\leq \frac{2 t^2 \log n_c (\log K)^{30}}{(\log\log n_c)^3} \leq \frac{1}{100}\rho^2 kx.$$
 Therefore, by applying Lemma~\ref{lem_units}, we get an $(2h_1,h_2,h_3)$-unit in $H-(W\cup U)$, a contradiction to the maximality of $U_1,\dots, U_m$. Hence we obtain $m\geq q$ as desired.
\end{proof}

\begin{proof}[Proof of Claim~\ref{cl: webs in c100 expander}]
 Let $U_1, \cdots, U_m$ be a maximal collection of internally vertex-disjoint 
 $(h_4,h_5,h_1,h_2,h_2)$-webs. Assume $m< h_4$, as otherwise we are done.
 Let $W$ be the set of all interior vertices of webs. Then we have
 \begin{align}\label{eq: W size c100}
	|W| & \leq h_4( 1 + h_4 h_5 + h_4 h_1h_3 ) \leq \frac{2 t^2 \log n_c (\log K)^{30} }{ (\log\log n_c)^{10}}.
 \end{align}
 Then by Claim \ref{cl: units in c100 expander}, we can find $q$ 
 vertex-disjoint $(2h_1,h_2,h_3)$-units $R_1, R_2, \cdots, R_q$.
Also, letting $R= \bigcup_{i\in [q]} V(R_i)$, 
we have
 \begin{align*}
	|W \cup R|  \leq |W| + q(1+2h_1h_3 + 2h_1h_2) 
	 \leq \frac{10 t^2 \log n_c \log^{30} K}{ (\log \log n_c)^2}.
\end{align*}
 Thus, by Claim~\ref{cl: stars in c100 expander}, $H-(W\cup R)$ contains $p$ vertex-disjoint copies $S_1, \cdots, S_{p}$ of $K_{1,2h_2}$ in $G-(W\cup R)$. 
 We will apply Lemma~\ref{lem_webs} with the parameters $n_H, p,q,x,h_1,h_2,h_3,h_4,h_5$ and $\rho$ play the roles of $\rho_1$ and $\rho_2$ at the same time and $\eta^{1/4}, W$ play the roles of $c$ and $W$, respectively.
 By our choice of parameters, all these numbers except $\rho_1$ and $\rho_2$ are larger than $T$.
We now check the conditions.
As we have 
 \begin{align*}
     h_1& = \log n_c \log\log n_c (\log K)^{20} \geq 80 \log n_c (\log K)^{10} =4 h_5,\\
     h_4&= \frac{\eta t}{(\log\log n_c)^6} \stackrel{\eqref{eq: log nHt2}}{\leq}   \eta^{1/4} \cdot \frac{\eps^2 \eta^{1/2} t}{\log^6 K} \leq \eta^{1/4} \rho^2x,\\
     q & = \frac{ t (\log K)^{10} }{(\log\log n_c)^4}\geq \frac{8 \eta t}{(\log\log n_c)^6} = 8h_4, \\
     h_5& \geq h_3 + 10 \log n_c (\log K)^{10}
     \geq   h_3 + \frac{50 \log q}{\rho_1} + \frac{10 \log\left(\frac{2n_H}{px} \right) + 10 \log\left(\frac{4n_H}{qh_1h_2} \right)}{\rho_2}.
 \end{align*}
We obtain the final inequality, as each of $q, \frac{2n}{px}, \frac{4n}{qh_1h_2}$ is smaller than $n_H$ and by \eqref{eq: log nc and log nH}.
 By \eqref{eq: W size c100}, we also have that 
 $$\rho px  =2 \eps \eta t^2 \log n_c \log^{96} K\geq 32(|W|+p).$$

Note that 
\begin{align*}
    2h_4h_5+p &=  \frac{2\eta t}{(\log \log d)^6} \times 20 \log n_c (\log K)^{10} + t \log n_c (\log K)^{100} \leq  \frac{2 t^2 \log n_c (\log K)^{30} }{ (\log\log n_c)^{10}}  \text{ and } \\
    \frac{1}{16} \rho qh_1h_2 &= \frac{\eps \eta^{1/2} t^2 \log n_c  (\log K)^{26}}{16(\log \log n_c)^3} \geq  \frac{4 t^2 \log n_c (\log K)^{30} }{ (\log\log n_c)^{10}}.
\end{align*}
Here, the final inequality holds by \eqref{eq: log nHt2}.
So we can conclude that $|W|+p+2h_4h_5 \leq \rho \min\{\frac{px}{8}, \frac{qh_1h_2}{16}\}$.

Hence, Lemma~\ref{lem_webs} provides another $(h_4, h_5, h_1, h_2, h_3)$-web $U_{m+1}$ internally disjoint from $U_1, \dots, U_m$, a contradiction to the maximality of $m$.
Therefore, $m \geq h_4$.
 \end{proof}

\end{document}